\documentclass[11pt]{article}
\usepackage{amsfonts}
\usepackage[square, numbers]{natbib}

\usepackage[left=0.8in,top=0.8in,right=0.8in,bottom=1.0in]{geometry}
\usepackage{graphics,epsfig}
\usepackage{color}
\usepackage{amsmath,amsthm,amssymb,amsfonts, mathrsfs, epsfig,latexsym,graphicx, enumerate}

\usepackage{macro}
\usepackage{appendix}

\newtheorem{theorem}{Theorem}[section]

\theoremstyle{definition}

\theoremstyle{remark}

\newtheorem{lemma}[theorem]{Lemma}

\newtheorem{proposition}[theorem]{Proposition}
\numberwithin{equation}{section}
\newtheorem{condition}[theorem]{Condition}

\def\O{\Omega}
\def\o{{\omega}}
\def\bA{{\bf A}}
\def\bB{{\bf  B}}

\def\l{\lambda}
\def\p{\partial}
\def\t{\times}
\def\e{\epsilon}
\def\s{^{\star}}
\def\f{{\mathbf f}}

\def\bu{\mathbf u}
\def\ubar{\bar{\mathbf u}}
\def\ue{{\bar{\mathbf u}}^{\epsilon}}

\def\bv{\mathbf v}
\def\bw{\mathbf w}
\def\bz{\mathbf z}
\def\by{\mathbf y}

\def\n{{\mathbf n}}
\def\a{\alpha}
\def\d{\delta}
\def\k{\kappa}
\def\th{\theta}
\def\E{\mathbb{E}}
\def\R{{\mathbb{R}}}

\def\sni{\smallskip\noindent}
\def\ni{\noindent}
\def\pr{^\prime}

\providecommand{\norm}[1]{\lVert#1\rVert}
\date{}
\begin{document}
\title{A Class of Stochastic Partial Differential Equations with Jumps in Fluid Dynamics: Large and Moderate Deviation Asymptotics}
\author{
\begin{tabular}{c@{\hspace{2cm}}c}
Arnab Ganguly\footnote{\textit{Email:} aganguly@math.lsu.edu. Research supported in part by NSF grants DMS-1855788 and DMS-2246815, and by the Simons Foundation (Travel Support for Mathematicians program). } &
Padmanabhan Sundar\footnote{\textit{Email:} psundar@lsu.edu.}
\end{tabular}\\[1ex]
Department of Mathematics\\
Louisiana State University
}

\maketitle
\begin{abstract}
We establish large and moderate deviation principles for a broad class of nonlinear stochastic partial differential equations driven by multiplicative Poisson random measures in the small-noise regime. Our abstract framework is formulated for locally monotone evolution equations over a Hilbert triple and encompasses several important models arising in fluid dynamics and turbulence theory, including the two-dimensional Navier--Stokes equations, magnetohydrodynamics (MHD), the GOY shell model of turbulence, the two-dimensional tidal equations, and Navier--Stokes equations with nonlinear viscosities. A key analytical ingredient is the well-posedness of the associated controlled equations, established using the method of local monotonicity. The large and moderate deviation principles are then proved via the weak convergence approach.

\end{abstract}



\section{Introduction}
\setcounter{equation}{0}
The asymptotic analysis of stochastic systems under small random perturbations has been a central topic in probability theory for several decades. In many applications, deterministic models capture only the average behavior of the underlying physical system, while random perturbations, though small, may produce rare but significant events. Quantifying the probabilities of such events is the objective of the theories of large and moderate deviations. In the context of stochastic partial differential equations (SPDEs), these questions are particularly important for mathematical models arising in fluid dynamics, turbulence, climate science, and other infinite-dimensional dynamical systems, where random perturbations naturally occur and rare events often correspond to physically meaningful phenomena.

Large deviation principles (LDPs) provide precise exponential estimates for the probabilities of deviations of stochastic trajectories from their deterministic limits. Such estimates play an important role in the study of metastability, exit problems, stochastic control, uncertainty quantification, and rare-event simulation. Moderate deviation principles (MDPs), on the other hand, describe deviations on scales intermediate between the central limit theorem (CLT) and the large deviation regime. Consequently, they furnish a refined asymptotic theory that captures fluctuations beyond the Gaussian approximation while retaining the tractable quadratic structure of the associated rate function. MDPs have found applications in asymptotic statistics, hypothesis testing, importance sampling, stochastic optimization, and statistical inference for stochastic systems.

In the present work, we establish both large and moderate deviation principles for a broad class of nonlinear SPDEs driven by multiplicative jump noise modeled by compensated Poisson random measures (PRMs). Our framework encompasses several important equations arising in fluid dynamics and turbulence theory, including the two-dimensional Navier--Stokes equations, magnetohydrodynamics (MHD), the GOY shell model of turbulence, the two-dimensional tidal equations, and related locally monotone systems.  The results provide a unified treatment of both large and moderate deviations for these models within a common abstract framework and, in particular, cover several equations for which such results are not currently available in the literature, including Navier-Stokes equations with nonlinear viscosities.


The theory of large deviations has a rich history, beginning with the pioneering works of Cram\'er \cite{Cr}, Freidlin and Wentzell \cite{FW}, Donsker and Varadhan \cite{DV75I, DV75II, Va} (see also \cite{DeZe10} for a general introduction to the subject). In the presence of small noise perturbations, finite-dimensional stochastic differential equations driven by Brownian motion of the form
\begin{equation}
dX^{\e}(t)=b(X^{\e}(t))\,dt+\sqrt{\e}\sigma(X^{\e}(t))\,dW(t),
\label{syst1}
\end{equation}
with $X^{\e}(0)=x_{\e}$, were studied extensively by Freidlin and Wentzell under suitable assumptions on the coefficients $b$ and $\sigma$. If $x_{\e}\to x_0$ as $\e\to0$, then $X^{\e}$ converges in probability to the solution $X^0$ of the deterministic ordinary differential equation
$
\frac{dX^0}{dt}=b(X^0(t)),  
$
with $X^{0}(0)=x_0$.
The Freidlin-Wentzell theory provides sharp exponential estimates for the probabilities that the stochastic trajectories deviate from their deterministic counterpart. Since its inception, the theory has been extended in numerous directions, including infinite-dimensional systems, stochastic evolution equations, delay equations, and systems driven by jump noise.

While the theory of LDP has been extensively developed over the past several decades, MDPs are of comparatively more recent origin. Moderate deviations for diffusion processes and, more generally, Markov processes were first investigated in \cite{LS, Gu, HS}. For a stochastic system $X^\ep$ (e.g. \eqref{syst1}) converging to a deterministic limit $X^0$, the MDP concerns the asymptotic behavior of the scaled process
$$
Y^{\e}(t)=\frac{X^{\e}(t)-X^0(t)}{a(\e)}, \qquad t\in[0,T],
$$
where the scaling function $a(\e)$ satisfies
$
\sqrt{\e}\ll a(\e)\ll1.
$
The choice $a(\e)=\sqrt{\e}$ corresponds to the central limit theorem, whereas $a(\e)=1$ yields the standard large deviation regime. Intermediate scalings, for example $a(\e)=\e^{1/3}$, lead to MDP, where the goal is to show that for a suitable rate function $I$,
$$\mathbb{P}(Y^\ep \in A) \approx e^{-\inf_{y \in A} I(y) / b(\ep)}.$$
Here
$
b(\e)\equiv \ep / a^2(\e),
$
and we refer to $1/b(\e)$ as the speed of the MDP.

A major advance in the modern theory of large deviations was the development of the variational representation and weak convergence approach by Budhiraja, Dupuis, and Ellis \cite{DE, BuDu19}. Classical proofs of LDPs often rely on discretization or approximation arguments together with the contraction principle. In contrast, the weak convergence approach exploits variational representations arising from stochastic control to establish the Laplace principle directly. For random elements taking values in Polish spaces, the Laplace principle is equivalent to the large deviation principle. Moreover, the same variational framework has proved equally effective in the study of moderate deviations.

The weak convergence approach provides a general set of sufficient conditions for establishing the Laplace principle for families of random elements taking values in Polish spaces. Since its introduction in the context of stochastic Navier-Stokes equation by by Sritharan and Sundar \cite{SS}, it has led to a substantial body of work on large deviation principles for infinite-dimensional stochastic systems driven by both Brownian and Poisson noise; see, for example, \cite{RZ, BDM08, BDM09, MSS, XZ, BCD, CD} for a short representative sample of this literature.
The corresponding weak convergence framework for moderate deviations of infinite-dimensional stochastic systems driven by Poisson random measures (PRMs) was developed by Budhiraja, Dupuis, and Ganguly \cite{BDG}. Compared with large deviations, moderate deviations require a different asymptotic scaling and substantially different estimates, and have consequently received much less attention in the literature particularly in the nonlinear settings. Among the few available results in this area are the works of Dong \emph{et al.} \cite{DXZ} and Wang \emph{et al.} \cite{WZZ}, which establish moderate deviation principles for the two-dimensional Navier--Stokes equations under suitable assumptions.

The present work employs the weak convergence approach to provide a unified framework for establishing both large and moderate deviation principles for nonlinear SPDEs driven by compensated PRMs. The equations are formulated in an abstract evolution setting over a Hilbert triple, allowing a unified treatment of a broad class of locally monotone systems arising in fluid dynamics, including the two-dimensional Navier--Stokes equations, MHD equations, Boussinesq systems, tidal equations, the GOY shell model of turbulence, and Navier-Stokes equations with nonlinear viscosities. A key analytical ingredient is the well-posedness of the associated controlled equations. We establish existence and uniqueness of these controlled systems by the method of local monotonicity, which provides the foundation for verifying the hypotheses of the weak convergence approach and thereby proving both the large and moderate deviation principles. The LDP and MDP results are presented in Theorem \ref{th:ldpu} and Theorem \ref{th:mdpu}.

From the viewpoint of fluid dynamics, rare events are closely related to the onset of extreme phenomena and the transition to turbulent behavior. One example is the occurrence of exceptionally large vorticity over a finite time interval,
\[
\left\{\omega:\sup_{0\le t\le T}\|\mathbf{u}(t,\omega)\|_V\ge n\right\},
\]
for large values of $n$. It is well known (cf.~\cite{BM2}) that the growth of the vorticity plays a fundamental role in the possible breakdown of smooth solutions to the three-dimensional Euler equations. Motivated by the work of Birnir \cite{Bi}, we model the random perturbations using jump noise. Although Brownian perturbations can be incorporated into our framework with relatively minor modifications, doing so introduces additional technicalities without yielding essentially new mathematical ideas.

Finally, we remark that while the controlled equations are naturally analyzed in a Lusin space, our main large and moderate deviation results are formulated under a Polish-space topology so that the Laplace principle is equivalent to the corresponding large or moderate deviation principle. The remainder of the paper is organized as follows. Section~2 introduces the abstract framework together with the formulation of the LDP and MDP. In Section~3 we establish the existence and uniqueness of solutions to the controlled SPDEs using the method of local monotonicity. Sections~4 and~5 are devoted to the proofs of the LDP and the MDP, respectively, via the weak convergence approach. The appendix contains several examples of nonlinear PDEs which when perturbed by a small jump noise 
(driven by a compensated PRM) fall within the scope of our general framework.

\section{Mathematical Framework}
\setcounter{equation}{0}

In this work, LDP and MDP will be established for a class of nonlinear stochastic differential equations driven by Poisson random measures. We present an abstract framework in order to study a class of stochastic partial differential equations  (SPDEs) that arise in fluid dynamics. 

 Suppose that $H$ and $V$    are two real separable Hilbert spaces such that the embedding $V \subset_{\to} H$ is  dense and compact. It enables one to write the following Hilbert triple (cf. Brezis \cite{Br})
\[
V \subset_{\to} H \simeq H^{*} \subset_{\to} V^{\prime}
\]
where $V\pr$ is the dual of $V$. 
We introduce an unbounded self-adjoint operator $\bA$  on $H$,  and a bilinear operator $\bB: V \times V \to V\pr$  which satisfy the conditions listed below. We will use the notation $| \cdot |$ and $(\cdot , \cdot)$ for the norm and inner 
product on the space $H$, and $\norm{ \cdot }$ and $((\cdot, \cdot))$ for the norm and inner 
product on the space $V$. The dual pairing between $V^{\prime}$ and $ V$ will be denoted by the sharp bracket $\langle\,,\,\rangle$. 

{\bf Model:} Consider a stochastic basis $(\Omega, \mathcal{F}, \{\mathcal{F}_t: t \in [0, T]\}, P)$ with the filtration 
$\{\mathcal{F}_t: t \in [0, T]\}$ satisfying the usual hypotheses.  Let $\mathbb Z$ be a locally compact Polish space, and for each $\ep>0$, let $N^{1/\e}$ be an $\mathcal{F}_t$-adapted  PRM defined on $\mathbb Z \times [0, T]$ with compensator measure $\e^{-1}\nu(dz) ds$ (cf. Appendix B).  We assume that the measure $\nu$ on $(\mathbb Z, \mathcal{B}(\mathbb Z))$ is Radon. Denoting the 
compensated PRM by $\tilde{N}$,   for each $\ep>0$, we consider a general class of SPDEs,
which in  integral form are given by 
\begin{align}
{\bu^{\e}}(t)  + \int_0^t \{\bA \bu^{\e} (s) + \bB(\bu (s))\}ds = {\bu}(0) + \int_0^t {\f}(s)  + {\e}\int_{\mathbb{Z}\times[0,t]}  g(s, {\bu}^{\e}(s), z)\,\tilde{N}^{\e^{-1}}(ds, dz). \label{epsiloneqn}
\end{align}
where $\bu(0)$ is in $H$, and $f \in L^2([0, T]; V\pr)$. Our primary goal in this paper is to analyze the asymptotics of deviation of $u^\ep$ from its limit $u^0$ given by 
\begin{align}\label{eq:limiteqn}
d\bu^0(t) + \big[\bA\bu^0(t) + \bB(\bu^0(t), \bu^0(t))\big] dt = \f(t) dt.
\end{align}

\smallskip
\noindent
{\bf Hypotheses H:}
\begin{enumerate}
\item[{\bf H.1}] The map $\bA: V \to V\pr$ is a linear operator such that for all $u, v \in V$, $\langle \bA u, v \rangle =  ((u, v))$. The operator $\mathbf{A}$ has eigenvalues
 $0<\lambda_{1}<\lambda_{2}\leq\cdots$ with the corresponding
eigenfunctions denoted by $e_{1},e_{2},\cdots$. The eigenfunctions form a complete
orthonormal system for ${H}$. 

\item[{\bf H.2}] The map $\bB: V \times V \to V\pr$ is a bilinear operator, and there exists a constant $c_1 > 0$ such that 
\[
|\bB (u, v)|^2_{V\pr} \le c_1 |u|\,\norm{u}\,|v| \,\norm{v}\,\,\text{for all}\,\, u, v \in V.
\]
\item[{\bf H.3}] The operator $\bB$ satisfies the condition 
\[ \langle \bB(u, v), w \rangle = - \langle \bB(u, w), v\rangle \,\,\text{for all}\,\, u, v, w \in V.
\]
\item[{\bf H.4}] Let $\bB(u, u)$ be denoted by $\bB(u)$. For every $r > 0$, there exists a constant $L_r$ such that 
\[
|\bB(u) - \bB(v)|_{V\pr} \le L_r\norm{u - v}\,\,\text{for all} \,\,u, v \in V,\,\,\text{and} \,\,\norm{u}, \norm{v} \le r.
\]
 \item[\bf H.5] The function $g$, called the noise coefficient,  maps $ [0, T]\times H \times \mathbb Z \to H$ and is
  jointly measurable. In addition, there exists an $L^2(\nu)$ function  $M_g(z) > 0$ such that 
\begin{equation}
 |g(t, u, z)| \le M_g(z) (1 + |u|) \,\,\text{for all} \,\,t \in [0, T]\,\, u \in H, \,\text{and}\,\, z \in \mathbb Z. \label{growth}
\end{equation}
\item[{\bf H.6}] For all $u_1, u_2 \in H$, and $t \in [0, T]$, there exists an $L^2(\nu)$ function $L_g(z) > 0$ such that 
\begin{equation}
 |g(t, u_1, z) - g(t, u_2, z)|  \le L_g(z) |u_1 - u_2| \,\,\text{for all}\,\, z \in \mathbb Z.\label{lip}
\end{equation}
\item[{\bf H.7}]  (Local Monotonicity) Define  a nonlinear operator $F$ on $V$ by
$F(\bu):= -\bA u - \bB(u) $.  There exists a weaker norm $\norm{\cdot}_0$  on $V$ (i.e.,  there exists a constant $C_0$ such that $\norm{\cdot}_0 \le C_0 \norm{\cdot}_V$),  a non-negative continuous function $h$, and an $\e > 0$ such that 
 for any $r > 0$, $\bu\in V$ and $\bv \in B_r$, viz. the $r$-ball in $(V, \norm{\cdot}_0)$,   such that 
 the pair $(F, \e g)$
is monotone, i.e.,
\begin{equation}
\langle F(\bu) - F(\bv), \bw \rangle  - h(r)\,|\bw|^2 + \e\int_{\mathbb Z}  |g(t, \bu, z) - g(t, \bv, z)|^2\nu(dz) \le 0\label{mono}
\end{equation}
for all $t \in [0, T]$, where $\bw :=  \bu - \bv$. 
\item[{\bf H.8}] There exists a $\d > 0$ such that for all $B \in \mathcal{B}(\mathbb Z)$ satisfying $\nu(B) < \infty$ the following holds:
\[
 \int_B \exp\{\d (M_g^2(z) + L_g^2(z))\} \nu(dz) < \infty.
\]
\end{enumerate}
{\sc Notation:} Throughout the paper, we use the notation  $M_k := \norm{M_g}_{L^k(\nu)}$, and $L_k := \norm{L_g}_{L^k(\nu)}$ for any $k \ge 1$, whenever they are finite.


{\bf Large and moderate deviations and Laplace Principle:} Let $\{Z^\ep\}$ be a family of random variables taking values in a Polish space $\mathbb{U}$. $\{Z^\ep\}$ satisfies an LDP with rate function $I: \mathbb U \rt [0,\infty]$ and speed $1/b(\ep)$, where $b(\ep) \stackrel{\ep \rt 0}\Rt 0$, if for any Borel set $A \subset \mathbb{U}$
\begin{equation}
\label{eq:ldp_exponential_rates:old}
\begin{aligned}	
	-I(A^\circ)~& \leq \liminf_{\ep\to 0}~b(\ep) \log \mathbb{P}\lf( Z^\ep \in A \ri)  \leq \limsup_{\ep\to 0}~b(\ep) \log \mathbb{P}\lf( Z^\ep \in A \ri) \leq -  I(\bar{A}).
\end{aligned}
\end{equation}
Here for $A \subset \mathbb{U}$, $A^\circ$ denotes its interior, $\bar A$ its closure, and $I(A) = \inf_{u \in A} I(u)$.

The LDP is equivalent to the {\bf Laplace principle} which is said to hold for $\{Z^\ep\}$ with rate function $I: \mathbb S \rt [0,\infty]$ and speed $1/b(\ep)$ if for any $h \in C_b(\mathbb{U}, \R)$
 \begin{equation}
 \lim_{\e \to 0} b(\e)\log\mathbb{E}\lf[\exp\lf(-h(Z^{\e}) / b(\e)\ri)\ri] = - \inf_{u \in \mathbb{S}}\{h(u) + I(u)\},\label{introineq1}
 \end{equation}

Suppose now $\mathbb{U}$ is also a vector-space, $X^\ep$ are $\mathbb{U}$-valued random variables with  $X^\ep \stackrel{\mathbb{P}}\Rt X$ where the limit $X$ is deterministic. Then establishing MDP of $X^\ep$ with scaling factor $\sqrt \ep \ll a(\ep) \ll 1$ refers to proving that an LDP holds for $Z^\ep \equiv (X^\ep-X)/a(\ep)$ with speed $1/b(\ep) \equiv a^2(\ep)/ \ep$ and a rate function $I$.

 We now state the main results of the paper.

{\bf LDP of $\bu^\ep$:} 
 Define $\ell :[0,\infty )\rightarrow \lbrack 0,\infty )$ by 
\begin{equation}
\ell (r)=r\log r-r+1,\quad r\in \lbrack 0,\infty ).  \label{eq:eq121}
\end{equation}%
For $\varphi: \mathbb Z \times [0,T] \rt [0,\infty)$, define
\begin{align}\label{Ltdef}
    L_t(\vphi) \doteq \int_{\mathbb Z \times [0,t]} \ell(\phi(z,s)) \nu(dz) ds
\end{align}
when $\ell\circ \vphi$ is integrable with respect to $\nu(du) ds$ and $\infty$ otherwise. Define 
\begin{align}
    \label{eq:defS}
    \mathbb{S} := \{\varphi: \mathbb{Z}\times [0,T] \to [0, \infty): L_T(\varphi) < \infty\}.
\end{align}

\begin{theorem}\label{th:ldpu}
 Let $\bu^\ep$ be the solution to \eqref{epsiloneqn}  with $\bu^\ep(0) = \xi \in H$. Then $\{\bu^\ep\}$ satisfies an  LDP in  
$D([0, T]; V^{\prime})$ with rate function $I^L: D([0, T]; V^{\prime}) \rt [0,\infty] $ given by
\begin{align}\label{rate-ldp}
I^{\mathrm{LD}}(\ubar) = \inf_{\varphi \in \mathbb{S}^{\mathrm{LD}}_{\ubar}} L_T(\varphi)
\end{align}
where  $\mathbb{S}^{\mathrm{LD}}_{\ubar} \subset \mathbb{S}$ is the class of all controls $\vphi \in \mathbb{S}$ such that
\begin{align*}
d\ubar(t) + [\ A\ubar(t) + B(\ubar(t), \ubar(t))] dt = [\f(t) + \int_{\mathbb Z} g(t, \ubar(t), z)(\vphi(t,z) - 1) \nu(dz)]dt,
\end{align*}
with the functional $L_T$ defined in \eqref{Ltdef}. 
\end{theorem}

{\bf MDP of $\bu^\ep$:} For $0 \le t \le T$ and $\e > 0$, define the process 
\begin{equation}\label{yepsilon}
Y^{\e}(t) = (\bu^\ep(t) - \bu^0(t)) / a(\e), 
\end{equation}
where $\bu^0$ is given by \eqref{eq:limiteqn}.

\begin{theorem}\label{th:mdpu}
The family $\{\bu^\ep\}$ satisfies a MDP with scaling factor $a(\ep)$ in  
$D([0, T]; V^{\prime})$ with speed $1/b(\e) \equiv a^2(\ep)/\ep$  and rate function $I^M: D([0, T]; V^{\prime}) \to [0,\infty]$ defined by
\begin{align*}
    I^\mathrm{MD}(\bar y) = \frac{1}{2}\inf_{\psi \in \mathbb S^{\mathrm{MD}}_{\bar\by}}\int_{\mathbb Z \times [0,T]} |\psi(z,t)|^2 dz ds,
\end{align*}
where for $\bar \by \in D([0, T]; V^{\prime})$, $\mathbb{S}^{\mathrm{MD}}_{\bar \by}$ is the set of all controls $\psi \in L^2(\mathbb{Z}\times [0,T], \nu\ot \leb)$ such that 
\begin{equation}
d\bar \by(t) + \bA \bar \by(t)dt + \left[ \bB(\bar \by(t), \bu^0) + \bB(\bu^0, \bar \by(t))\right]dt =   \int_{\mathbb Z} g(t, \bu^0(t), z)\psi(t, z) \nu(dz)dt. \label{uzeroeqn}
\end{equation}
In other words, $\{Y^\ep\}$, defined by \eqref{yepsilon}, satisfies an LDP in  
$D([0, T]; V^{\prime})$ with speed $1/b(\e)$  and the rate function $I^{\mathrm{MD}}: D([0, T]; V^{\prime}) \to [0,\infty]$.
\end{theorem}

\section{Existence and Uniqueness of Solutions to the SPDE}
\setcounter{equation}{0}
Our aim in this section is to establish {\it a priori} bounds for 
solutions to \eqref{epsiloneqn}. Since the results in this section will be stated for a fixed $\ep$, for notational convenience, we denote $\bu^\ep$ simply by $\bu$. 

We start by stating the following useful result on relative compactness that underpins much of the theoretical development. Suppose that $V, H$ and $V^{\prime}$ are Hilbert spaces with $V^{\prime}$ being the dual of $V$. Let 
$V \subset_{\to} H$ be a compact, dense imbedding so that one obtains the Hilbert triple   $V \subset_{\to} H\subset_{\to} V^{\pr}$.  The following result is taken from Metivier ( see \cite{Met}, page 113) which uses the ideas of a result due to Dubinsky given in Lions ( see \cite{Li}). 
 
\begin{theorem}\label{tightness}
Let $q \ge 1$. Every subset of $L^q(0, T; H) \cap D([0, T]; V^{\prime})$ which is relatively compact in $D([0, T]; V^{\prime})$ 
and bounded in $L^q(0, T; V)$ is relatively compact in $L^q(0, T; H)$.
\end{theorem}

To solve \eqref{epsiloneqn}, we form the sequence of Galerkin approximations to it as follows. 
Let 

$H_n := {\rm Span} \{e_1, e_2, \cdots , e_n\}$ where $\{e_j\}$ is taken as the orthonormal basis in $H$ consisting of eigenfunctions of the operator ${\bA}$  so that each $e_j \in D(\bA)$. Let $P_n$ denote the orthogonal projection of $H$ to $H_n$.
Define ${\bu}^\ep_n$ as the solution of the following stochastic differential equation: For each $\bv \in H_n$,
\begin{align}
&d({\bu}_n(t) , {\bv}) + \{\langle \bA {\bu}_n(t) + \bB({\bu}_n(t)), \bv \rangle\}dt \notag\\
& = \langle {\f}(t), {\bv} \rangle  + {\e}\int_{\mathbb Z} ( g({\bu}_n(t), z), \bv)\,\tilde{N}^{\e^{-1}}(dt, dz) 
\label{eqn3.2}
\end{align}
with $({\bu}_n(0), \bv)= ({\bu}(0), \bv)$. The coefficients are locally Lipschitz, and hence we need  {\it{a priori}} estimate 
to claim global existence of a solution $\bu_n$ as an adapted process in $C(0, T; H_n)$. Indeed, \eqref{eqn3.2} implies that 
\begin{align}
&d{\bu}p_n(t) + \{\bA_n {\bu}_n(t) + \bB_n({\bu}_n(t))\}dt \notag\\
& = {\f}_n(t) + {\e}\int_{\mathbb Z} g_n({\bu}_n(t), z) \,\tilde{N}^{\e^{-1}}(dt, dz)
\label{eqn3.3}
\end{align}
where $\bA_n, \bB_n, \f_n$ and $g_n$ are the orthogonal projections of $\bA, \bB, \f$ and $g$ to 
the finite dimensional space $H_n$.

\begin{proposition}\label{prop3.1}
 Let  $E|{\bu}_0|^2 < \infty$ and ${\f} \in L^2(0, T; V\pr)$. Then, under hypotheses {\bf{H}}, we have the following estimates:
\begin{enumerate}
 \item[(i)] There exists a positive constant $C_1$ that depends on $E|\bu_0|^2, \int_0^T \mathbb{E}|\f (s)|_{V\pr}^2ds$, and $T$ such that
\begin{equation}
 \sup_{0 \le t \le T} \mathbb{E}\left(|{\bu}_n(t)|^2\right) + \int_0^T \mathbb{E} \left(\norm{{\bu}_n(s)}^2\right)ds \le 
 C_1 \label{est1}
\end{equation}
and
\item[(ii)] there exists a positive constant $C_2$ that depends on $E|\bu_0|^2, \int_0^T \mathbb{E}|\f (s)|_{V\pr}^2ds$, and $T$ such that 
\begin{equation}
 \mathbb{E}[\sup_{0 \le t \le T} |\bu_n (t)|^2] \le C_2 \label{est2}
\end{equation}
\end{enumerate}
\end{proposition}
The proof of the Proposition is omitted since an analogous result will be proved for more general equations in Section 4 (see Proposition \ref{scontrol}). 


We proceed to  obtain  higher order a priori estimates in the following proposition. 
\begin{proposition}\label{prop3.2}
Assume Conditions ${\bf H.1}$ - ${\bf H. 6}$. In addition, assume that $E|\bu (0)|^4 < \infty$, $\f$ be in $L^4( [0, T] : V^{\pr})$, and the function $M_g$ in Hypothesis {\bf H.5} be in $L^i(\nu)$ for $i = 2, 3, 4$.  With the notation 
 $M_i = \norm{M_g}_{L^i(\nu)}$, for all $0 \le t \le T$
and $\d > 0$, there exists a positive constant $C$ that depends on $\mathbb{E}|\bu(0)|^4, \int_0^{T}e^{-\d s}\,\norm{\f(s)}_{V\pr}^4ds, \d, M_2, M_3, M_4$ and $T$ such that
\begin{align}
&\mathbb{E} \left(e^{-\d t} |\bu_n(t-)|^4\right)  +  \mathbb{E} \left(\int_0^{t} e^{-\d s}|\bu_n(s)|^2 \norm{\bu_n(s)}^2 ds\right)\notag\\
& \le C. \label{three}
\end{align}
\end{proposition}
\begin{proof}
Define $\tau_R = \inf\{t: |\bu_n(t)|^2 + \int_0^t\,||\bu_n(s)||^2\,ds  > R\}$. 
Applying the It\^o formula to the function $\varphi(t, x) = e^{-\d t}|x|^4$
and then taking expectation, we obtain
\begin{align}
& \mathbb{E} \left(e^{-\d (t\wedge \tau_R)} |\bu_n(t \wedge \tau_R)|^4\right) +  \mathbb{E}\left(\int_0^{t\wedge \tau_R}e^{-\d s} \left[\d |\bu_n(s)|^4 + 4 |\bu_n(s)|^2 \norm{\bu_n(s)}^2
\right]ds\right) \notag\\
&= \mathbb{E}\left(|\bu_n(0)|^4\right) + 4 \mathbb{E}\left(\int_0^{t\wedge \tau_R}e^{-\d s} |\bu_n(s)|^2 \langle \f(s), \bu_n(s)\rangle ds\right) \notag\\
&\;\;\;+ \mathbb{E}\Big(\int_0^{t\wedge \tau_R}\int_{\mathbb Z} e^{-\d s} [|\bu_n(s-) + \e g_n(s, \bu_n(s-), z)|^4 - |\bu_n(s-)|^4 \notag\\
&\;\;\;\;\;\;\;\;\;\;\;\;\;\;\;\;\;\;- 4|\bu_n(s)|^2 (\bu_n(s-), \e g_n(s,  \bu_n(s-), z))]\frac{\nu (dz)}{\e}ds\Big)\label{threepr}
\end{align}
The second term on the right is dominated by 
\begin{align}
&  2\mathbb{E}\left(\int_0^{t\wedge \tau_R}e^{-\d s} |\bu_n(s)|^2\{ \norm{\bu_n(s)}^2 + \norm{\f(s)}_{V\pr}^2\}ds\right)\notag\\
&\le 2 \mathbb{E}\left(\int_0^{t\wedge \tau_R}e^{-\d s} |\bu_n(s)|^2\norm{\bu_n(s)}^2ds\right) + \frac{1}{\d}\mathbb{E}\left(\int_0^{t\wedge \tau_R}e^{-\d s}\,\norm{\f(s)}_{V\pr}^4ds\right)\notag\\
&\;\;\;\;\; + \frac{\d}{2} \mathbb{E}\left(\int_0^{t\wedge \tau_R}e^{-\d s} |\bu_n(s)|^4ds\right).\label{threedpr}
\end{align}
The third term on the right side of \eqref{threepr} admits for a suitable constant $\tilde{C}_1 = C(T, \d, M_2, M_3, M_4)$ the upper bound 
\begin{equation}
   \tilde{C}_1 \e\left(\mathbb{E}\int_0^{t \wedge \tau_R} e^{- \d s} (1 + |\bu_n(s)|^4) ds\right)\label{threetpr}
 \end{equation}
 by the growth condition on $g$, and the finiteness of $M_g$ in $L^i(\nu)$ for $i = 2, 3, 4$.
 
 The  a priori estimates obtained in Proposition \ref{prop3.1} and the bounds \eqref{threedpr} and \eqref{threetpr} when used in  \eqref{threepr}, yield by a Gronwall argument that  there exists a constant $C_3 > 0$ such that
\begin{align}
& \mathbb{E} \left(e^{-\d (t\wedge \tau_R)} |\bu_n(t \wedge \tau_R)|^4\right)  +  \mathbb{E} \left(\int_0^{t \wedge \tau_R} e^{-\d s}|\bu_n(s)|^2 \norm{\bu_n}^2 ds\right) \notag\\
&\le C_3 \label{threedone}
\end{align}
The proof is over by observing that $\tau_R \to \infty$ a.s. upon letting $R \to \infty$.
\end{proof}
\begin{proposition}
 Under the hypotheses of Proposition \ref{prop3.2}, there exists a constant $C_4 > 0$ such that the following estimate holds for any $T > 0$:
\begin{align}
 &\mathbb{E}\left( \sup_{0 \le t \le T}e^{-\d t} |\bu_n(t-)|^4  +  \mathbb{E} \int_0^{T} e^{-\d s}|\bu_n(s)|^2 \norm{\bu_n(s)}^2 ds\right) \le C_4.\label{four}
\end{align}
\end{proposition}
\begin{proof}
 The proof follows by the estimates given in the previous proposition and Burkholder-Davis-Gundy inequality.
\end{proof}

With the above estimates, we prove the existence and uniqueness of solutions to the stochastic equation.
\begin{theorem}
In addition to the hypotheses of Proposition 3.2, let $\e > 0$ such that hypothesis {\bf H.7} holds. Then there exists a unique adapted process $\bu$ taking values in 
$D([0, T]; H) \cap L^2(0, T; V)$ that satisfies the a priori bounds given in the previous propositions and solves the stochastic equation \eqref{epsiloneqn} with parameter $\e$.
\end{theorem}
\begin{proof}
 For the stochastic equation with parameter $\e$, let us call the 
$n$th Galerkin approximation as $\bu_n$,  omitting the dependence on $\e$ since the latter is fixed throughout the proof. 

Let $F(\bu): = - \bA(\bu) - \bB(\bu)$.
Using the a priori estimates, it follows from the Banach-Alaoglu theorem
that along a subsequence, the Galerkin approximations
$\{\bu_n\}$ have the following limits as $n \rt \infty$:
\begin{align}
&\bu_n\longrightarrow \bu \quad  \text {weakly in}\, \mathrm{L}^4(\Omega ;
  \mathrm{L}^{\infty}(0,T; H)) \cap\mathrm{L}^{2}
  (\Omega;\mathrm{L}^{2}(0,T; V)),\nonumber\\
& F(\bu_n)\longrightarrow F_0\quad \text{weakly in}\
\mathrm{L}^{2}(\Omega;\mathrm{L}^{2}(0,T; V^{\prime})),\nonumber\\
& g_n(\cdot, \bu_n, \cdot) \longrightarrow G\quad \text{weakly in} \,\mathrm{L}^{2}(\Omega;\mathrm{L}^{2}([0,T]\times Z; H)).
\end{align}
The assertion of the second statement holds since $F(\bu_n)$ is bounded in $\mathrm{L}^{2}(\Omega;\mathrm{L}^{2}(0,T; V^{\prime}))$. Likewise since noise coefficient has linear growth and $\bu_n$ is bounded in $\mathrm{L}^2(\Omega; L^2(0,T; H))$ uniformly in $n$, the last statement holds. Then $\bu$ has the It\^{o} differential
\begin{align}
d\bu(t) = F_0(t) dt + {\e} \int_{\mathbb Z} G(t, z)  \tilde{N}^{\e^{-1}}(dt, dz)\quad \text{weakly in}\
\mathrm{L}^{2}(\Omega;\mathrm{L}^{2}(0,T; V^{\prime})).\label{limit}
\end{align}
Fix an arbitrary adapted process $\bv$  in $L^{\infty}(\Omega \times (0, T); H_m)$ for $m \le n$, and define  
\begin{equation}\label{9}
r(t) \doteq  \int_0^t h(\|\bv(s)\|_0)  ds
\end{equation}
where  the function $h$ is as in {\bf H.7}.
 Then applying the It\^{o} Lemma to the function $e^{-r(t)} |\bu_n(t)|^2$, 
\begin{align*}
d\big[e^{-r(t)} |\bu_n(t)|^2\big] &= e^{-r(t)}\big(2F(\bu_n(t)) - \dot{r}(t) \bu_n(t), \bu_n(t)\big) dt + \e^2 e^{-r(t)}\int_{\mathbb Z}  |g_n(t, \bu_n(t), z)|^2 N^{\e^{-1}}(dt, dz)  \\
& \quad + 2 {\e} e^{-r(t)}\int_{\mathbb Z} \big(g_n(t, \bu_n(t-), z), \bu_n(t-)\big)  \tilde{N}^{\e^{-1}}(dt, dz).
\end{align*}
Integrating between $0 \leq t\leq T$ and taking expectation,
\begin{align*}
\E\big[e^{-r(T)} |\bu_n(T)|^2 - |\bu_n(0)|^2\big] =& \E\big[\int_0^T e^{-r(t)}\big(2F(\bu_n(t)) - \dot{r}(t) \bu_n(t), \bu_n(t)\big)dt\big]\\& +\e \E\int_0^T e^{-r(t)}\int_{\mathbb Z}  |g_n(t, \bu_n(t), z)|^2 \nu(dz)dt.
\end{align*}
 By lower semicontinuity, weak convergence and \eqref{limit},
\begin{align}
&\liminf_n \E\big[\int_0^T e^{-r(t)}\big(2F(\bu_n(t)) - \dot{r}(t) \bu_n(t), \bu_n(t)\big) dt  +\e \int_0^T e^{-r(t)} \int_{\mathbb Z} |g_n(t, \bu_n(t), z)|^2 \nu(dz) dt\big]\nonumber\\
& = \liminf_n E\big[e^{-r(T)} |\bu_n(T)|^2 - |\bu_n(0)|^2\big]\nonumber\\
& \geq \E\big[e^{-r(T)} |\bu(T)|^2 - |\bu(0)|^2\big]\nonumber\\
& = \E\big[\int_0^T e^{-r(t)}\big(2F_0(t) - \dot{r}(t) \bu(t), \bu(t)\big) dt +\e \int_0^T e^{-r(t)} |G(t, z)|^2 \nu(dz) dt \big].\label{lsc}
\end{align}
By local monotonicity,
\begin{align*}
& 2\E\big[\int_0^T e^{-r(t)} \big(F(\bu_n(t)) - F(\bv(t)), \bu_n(t) - \bv(t)\big) dt\big] - \E\big[\int_0^T e^{-r(t)} \dot{r}(t)|\bu_n(t) - \bv(t)|^2 dt\big]  \\
&\quad + \e\, \E\big[\int_0^T e^{-r(t)} \int_{\mathbb Z} |g_n(t, \bu_n(t), z) - g_n(t, \bv(t), z)|^2 \nu(dz)dt\big] \leq 0.
\end{align*}
Rearranging the terms,
\begin{align*}
& \E\big[\int_0^T e^{-r(t)}\big(2F(\bu_n(t)) - \dot{r}(t) \bu_n(t), \bu_n(t)\big) dt +\e \int_0^T e^{-r(t)} \int_{\mathbb Z} | g_n(t, \bu_n(t), z)|^2 \nu(dz) dt\big]\\
&\leq \E\big[\int_0^T e^{-r(t)}\big(2F(\bu_n(t))-\dot{r}(t)(2\bu_n(t) - \bv(t)), \bv(t)\big) dt\big] \\
& \,\,\,+ \E\big[\int_0^T e^{-r(t)}\big(2F(\bv(t)), \bu_n(t)-\bv(t)\big) dt\big] \\
&\,\,\,+ \e\, \E\big[\int_0^T e^{-r(t)}\int_{\mathbb Z} \big(2 g_n(t, \bu_n(t), z)- g_n(t, \bv(t), z), g_n(t, \bv(t), z)\big) \nu(dz)dt\big].
\end{align*}
Taking limit in $n$, using the result from \eqref{lsc} and rearranging, we obtain
\begin{align}
&E\big[\int_0^T e^{-r(t)}\big(2F_0(t)-2F(\bv(t)), \bu(t)-\bv(t)\big) dt\big] - E\big[\int_0^T e^{-r(t)}\dot{r}(t) |\bu(t) - \bv(t)|^2 dt\big] \notag\\
&\quad + \e \E\big[\int_0^T e^{-r(t)} \int_{\mathbb Z}  |G(t, z)- g(t, \bv(t), z)|^2 \nu(dz)dt\big] \le 0.\label{existence}
\end{align}
The above inequality holds by a density argument for any adapted process $\bv$ in $L^6(\Omega; L^{\infty}(0, T: H)) \cap L^2(\Omega; L^2(0, T; V))$ with 
\[
\mathbb{E} \{\int_0^T \norm{v(t)}_{L^4(G)}^4|v(t)|^2dt\} < \infty.
\]
Thus $\bu$ satisfies \eqref{existence}.

Notice that for $\bv=\bu$, $G(t)= g(t, \bu(t))$. Take $\bv = \bu - \lambda\bw$ with $\lambda >0$ and $\bw$ is an adapted process in $\mathrm{L}^2(\Omega; C(0, T; H) \cap \mathrm{L}^2(0, T; V))$ Then,
\begin{align*}
&\lambda E\big[\int_0^T e^{-r(t)}\big(2F_0(t)-2F(\bu - \lambda\bw)(t), \bw(t)\big) dt + \lambda\int_0^T e^{-r(t)}\dot{r}(t)|\bw(t)|^2 dt\big] \leq 0.
\end{align*}
Dividing by $\lambda$ on both sides of the inequality above and letting $\lambda$ go to $0$, one obtains
\begin{align*}
E\big[\int_0^T e^{-r(t)}\big(F_0(t)-F(\bu(t)), \bw(t)\big) dt\big] \leq 0.
\end{align*}
Since $\bw$ is arbitrary, we conclude that $F_0(t)=F(\bu(t))$. Thus the existence of a strong solution of the stochastic equation  has been proved.

If $\tilde\bu\in \mathrm{L}^2(\Omega; C(0, T; H) \cap \mathrm{L}^2(0, T; V))$ be another solution of the equation,  then $\bw=\bu -\tilde\bu$ solves the stochastic differential equation in $\mathrm{L}^2(\Omega;\mathrm{L}^2(0, T; V^{\prime}))$,
\begin{equation}\label{8}
d\bw(t) = (F(\bu(t)) - F(\tilde\bu(t)))dt + {\e}( g(t, \bu(t)) - g(t, \tilde\bu(t))) \tilde{N}^{\e^{-1}}(dt, dz).
\end{equation}
We denote $g_d = g(t, \bu(t), z) - g(t, \tilde\bu(t), z)$.

We now apply It\^{o} Lemma to the function $e^{-r(t)} |\bw(t)|^2$
and using  local monotonicity, we
get
\begin{align}
e^{-r(t)} d|\bw(t)|^2  &\leq  h(r) e^{-r(t)} |\bw(t)|^2  dt +  e^{-r(t)} \int_{\mathbb Z} \{2\e(g_d, \bw(t)) + \e^2|g_d|^2\}\tilde{N}^{\e^{-1}}(dt, dz) .
\end{align}
Thus we have
\begin{equation*}
d(e^{-r(t)}|\bw(t)|^2) \leq  h(r(t)) e^{-r(t)} |\bw(t)|^2 dt + \, e^{-r(t)} \int_{\mathbb Z} \{2\e (g_d, \bw(t)) + \e^2|g_d|^2\}\tilde{N}^{\e^{-1}}(dt, dz).
\end{equation*}
Finally integrating in $0\leq t\leq T$, taking expectation on both sides, we obtain P-a.s.
\begin{align*}
E\big[e^{-r(t)}|\bw(t)|^2\big] \leq C E|\bw(0)|^2,
\end{align*}
which assures the uniqueness of the strong solution.
\end{proof}

\section{Proof of LDP}
 \setcounter{equation}{0}
The aim of this section is to prove the LDP for $\{\bu^\ep : \e>0\}$ in
$D([0, T]; V^{\prime})$. To this end, we first summarize the general result from \cite{BDM09} that forms the backbone of the proof.

{\bf A general LDP result for functionals of PRM \cite{BDM09}}: Let $\mathbb{X}_T := \mathbb{Z} \times [0,T]$. For $M>0$, define $S^M := \{\varphi: \mathbb{X}_T \to [0, \infty): L_T(\varphi) \le M\}$. A function $\varphi \in S^M$ can be identified with a measure $\nu_T^{\varphi} \in \mathbb{M}$, defined by
\[
 \nu_T^{\varphi}(B) = \int_B \varphi(s, x) \nu(dx) ds, \;\;\forall \;\;B \in \mathcal{B}(\mathbb{X}_T). 
\]
Thus the usual topology of weak convergence of measures induces a topology on $S^M$ through this identification  under which $S^{M}$ is a compact space.  Notice that $\mathbb{S} = \cup_{M \ge 1}S^M$ (c.f. \eqref{eq:defS}).

Now let $\mathcal{P}$ be the predictable $\sigma$-field on $\Omega \times [0,T]$, that is, the $\sigma$-field generated by left-continuous adapted processes. Define $\bar{\mathcal{A}}_{b}$ as the set  controls $\vphi: \Omega \times \mathbb{X}_T \to [0,\infty)$ which are $\mathcal{P} \times \mathcal{B}(\mathbb Z) / \mathcal B([0,\infty))$-measurable and bounded. 

 Observe that the strong solution $\bu^\ep$  of \eqref{epsiloneqn} admits a representation $\bu^\ep = \mathcal{G}^{\e}\lf(\e N^{\e^{-1}}\ri)$
for a suitable Borel-measurable function $\mathcal{G}^{\e}: 
\mathbb{M} \to D([0, T]; V^{\prime})$. By the Girsanov's theorem it follows that for any finite $M > 0$ and $\varphi \in \mathcal{U}^M$, $\ue_{\varphi}  \doteq  \mathcal{G}^{\e}\lf(\e N^{\ep^{-1}\varphi}\ri)$
 is the unique strong solution of 
\begin{align}\label{11}
&d\ue_{\varphi} + \!\!\big[\nu \bA\ue_{\varphi} + \bB(\ue_{\varphi}, \ue_{\varphi})\big] dt = \!\{\f(t) + \!\!\!\int_{\mathbb Z} g(t, \ue_{\varphi}, z)(\varphi(t, z) - 1)\nu(dz)\}dt + \e\!\!\int_{\mathbb Z} g(t, \ue_{\varphi}, z) \tilde{N}^{\e^{-1}\varphi}(dt, dz)\\
& \text{with}\,\,\ue(0) = \xi \,\,\text{in}\,\, \mathrm{L}^2(\Omega; X), \,\text{where}\,\, X=D(0, T; V^{\prime}) \cap \mathrm{L}^2(0, T; V).\nonumber
\end{align}

%
%
%

For fixed $\e$ and $\varphi$ fixed, let $\ue_{\varphi,n}$ denote the $n$th Galerkin approximation of the solution, $\ue_{\varphi}$, of \eqref{11}.  Then we have the following estimate.

\begin{lemma}\label{scontrol}
There exists a positive constant $C_1(E|\bu_0|^2, \int_0^T \mathbb{E}|\f (s)|_{V\pr}^2ds, T)$ such that
\begin{equation}
  \sup_{0\leq\ep\leq 1}\sup_{\varphi \in \mathcal{U}^M}\lf\{\mathbb{E}\left(\sup_{0 \le t \le T}|\ue_{\varphi,n}(t)|^2\right) + \int_0^T \mathbb{E} \left(\norm{\ue_{\varphi,n}(s)}^2\right)ds\ri\} \le  C_1\label{sconteqnest}
\end{equation}
\end{lemma}

{\em Notation:} The {\it overbar} symbol will be used to denote controlled processes.  For notational convenience, in proofs of the above and subsequent results, we will often drop the control $\vphi$ from the notation and simply denote the controlled processes $\ue_{\varphi}, \ue_{\varphi,n}$ respectively by $\ue, \ue_{n}$.

\begin{proof}
 We use the notation $K := 2M_2^2$ in what follows. In equation \eqref{11}, applying  the It\^o formula to the function $|\ue_{n}|^2$ yields
\begin{align}
& (d|\ue_{n}(t)|^2  + 2 ||\ue_{n}(t)||^2dt)\notag \\
&= 2<\f (t),\ue_{n}(t)> dt + 2\int_{\mathbb Z} (g_n(t, \ue_{n}(t), z), \ue_{n}(t))
(\varphi(t, z) - 1)\nu(dz)dt\notag\\
&+ \int_{\mathbb Z} \{2\e (\ue_{n}(t- ), g_n(t, \ue_{n}(t-), z)) + \e^2 |g_n(t, \ue_n(t-), z)|^2\} \tilde{N}^{\e^{-1}\varphi}(dt, dz)\notag\\&
+ \e\int_{\mathbb Z} |g_n(t, \ue_n(t-), z)|^2 \varphi(t, z)\nu(dz)dt. \label{eqn3.4}
\end{align}
Define $\tau_R = \inf\{t: |\ue_{n}(t)|^2 +  \int_0^t\,||\ue_{n}(s)||^2\,ds  > R\}$.
In integral form, \eqref{eqn3.4} yields after an application of Young's inequality, 
\begin{align}
&|\ue_{n}(t\wedge \tau_R)|^2 + 2 \int_0^{t\wedge \tau_R}||\ue_{n}(s)||^2ds\nonumber\\
& \le  |\ue_{n}(0)|^2 + \int_0^{t\wedge \tau_R}(|\f(s)|^2_{V^{\pr}} +  ||\ue_{n}(s)||^2)ds + 2\int_0^{t\wedge \tau_R}\!\!\int_{\mathbb Z} (g_n(s, \ue_{n}(s), z), \ue_{n}(s))
(\varphi(s, z) - 1)\nu(dz)ds\notag\\
&+ \int_0^{t\wedge \tau_R}\!\!\int_{\mathbb Z} \{2\e (\ue_{n}(s-), g_n(t, \ue_{n}(s-), z)) + \e^2 |g_n(s, \ue_{n}(s-), z)|^2\} \tilde{N}^{\e^{-1}\varphi}(ds, dz)\notag\\
&
+ \e \int_0^{t\wedge \tau_R}\!\!\int_{\mathbb Z} |g_n(s, \ue_{n}(s), z)|^2 \varphi(s, z)\nu(dz)ds.
\end{align}
Fix any $t_1 \in (0, T]$, and take supremum in the above inequality over $t$ in the interval $[0, t_1]$. Using the growth condition {\bf{H.5}} on $g$,  we have
\begin{align}
&\sup_{0 \le t \le t_1} \{|\ue_{n}(t\wedge \tau_R)|^2 +  \int_0^{t\wedge \tau_R}||\ue_{n}(s)||^2ds\}\notag\\
&\le |\bu_n(0)|^2 + \int_0^{T}|\f(s)|^2_{V^{\pr}}ds 
+ 2\int_0^{t_1\wedge \tau_R}|\ue_{n}(s)|\int_{\mathbb Z} M_g(z)(1 + |\ue_{n}(s)|)
|\varphi(s, z) - 1|\nu(dz)ds\notag\\
&+  \sup_{0 \le t \le t_1}\bar M^\ep_{t\wedge \tau_R} + 
\e\int_0^{t_1\wedge \tau_R}\int_{\mathbb Z} M_g^2(z)(1 + |\ue_{n}(s)|)^2 \varphi(s, z)\nu(dz)ds.\label{zero}
\end{align}
where 
\[
\bar M^\ep_{t}  = \int_0^{t}\!\!\int_{\mathbb Z} \{2\e (\ue_{n}(s-), g_n(s, \ue_{n}(s-), z)) + \e^2 |g_n(s, \ue_{n}(s-), z)|^2\} \tilde{N}^{\e^{-1}\varphi}(ds, dz).
\]

The third term on the right side of \eqref{zero} is bounded above  by
\begin{align}
&4\int_0^{t_1\wedge \tau_R}(1 + \sup_{0 \le r \le s}|\ue_{n}(r)|^2)
\int_{\mathbb Z} M_g(z)|\varphi(s, z) - 1|\nu(dz) ds\notag\\
\le & 4C + 4\int_0^{t_1}\sup_{0 \le r \le s}|\ue_{n}(r\wedge \tau_R)|^2\int_{\mathbb Z} M_g(z)|\varphi(s, z) - 1|\nu(dz) ds\label{3rd}
\end{align}
where $C := (\sup_{\varphi \in S^M}\int_0^T\int_{\mathbb Z} M_g(z)|\varphi(s, z) - 1|\nu(dz)ds < \infty$ by part (ii) of Lemma \ref{MLlemma}.
An analogous upper bound for the last term on the right side of equation \eqref{zero} is obtained along similar lines except that part (i) of Lemma \ref{MLlemma} is used. Indeed, if we set 
\[
 D := \sup_{\varphi \in S^M}\int_0^T\int_{\mathbb Z} M_g^2(z)\varphi(s, z)\nu(dz)ds,
\]
then the last term in \eqref{zero} viz., 
\begin{align}
 &\e\int_0^{t_1\wedge \tau_R}\int_{\mathbb Z} M_g^2(z)(1 + |\ue_{n}(s)|)^2 \varphi(s, z)\nu(dz)ds\notag\\
\le & 2\e D + 2\e \int_0^{t_1\wedge \tau_R}\sup_{0 \le r \le s}|\ue_{n}(r\wedge \tau_R)|^2 \int_{\mathbb Z} M_g^2(z)\varphi(s, z)\nu(dz)ds.
\label{3rdpr}
\end{align}

Using the bounds \eqref{3rd} and \eqref{3rdpr} above in \eqref{zero}, one obtains
\begin{align}
 &\sup_{0 \le t \le t_1} \{|\ue_{n}(t\wedge \tau_R)|^2 \notag\\
\le & |\bu_n(0)|^2 + \int_0^{T}|\f(s)|^2_{V^{\pr}} ds + 4C + 2 \e D + 4 \int_0^{t_1} 
\sup_{0 \le r \le s} |\ue_{n}(r\wedge \tau_R)|^2 \int_{\mathbb Z} M_g(z)|\varphi(s, z) - 1|\nu(dz)ds\notag\\
& + 2\e\int_0^{t_1} 
\sup_{0 \le r \le s} |\ue_{n}(r\wedge \tau_R)|^2 \int_{\mathbb Z} M_g^2(z)\varphi(s, z)\nu(dz)ds
+  \sup_{0 \le t \le T} \bar M^\ep_{t\wedge \tau_R}. \label{gwall}
\end{align}
Let $C_1 := 4C + 2 \e D$. By Gronwall Lemma, it follows that 
\[
 \sup_{0 \le t \le t_1} \{|\ue_{n}(t\wedge \tau_R)|^2  \le \{|\bu_n(0)|^2 + \int_0^{T}|\f(s)|^2_{V^{\pr}} ds + C_1 + \sup_{0 \le t \le T} \bar M^\ep_{t\wedge \tau_R} \} e^{C_1t_1}.
\]
Setting $t_1 = T$ and taking expectation, 
\begin{align}
 &\mathbb{E} \left[\sup_{0 \le t \le T} |\ue_{n}(t\wedge \tau_R)|^2 \right]
\le \left[|\bu_n(0)|^2 + \int_0^{T}|\f(s)|^2_{V^{\pr}} ds + C_1 +  \mathbb{E}\left(\sup_{0 \le t \le T} \bar M^\ep_{t\wedge \tau_R}\right)\right] e^{C_1T}. \label{gronwall}
\end{align}
First, write 
$\bar M^\e_{t} 
=  L_{t} + N_{t}$\
where 
\[
 L_t = 2\e\int_0^{t}\int_{\mathbb Z}  (\ue_{n}(s-), z), g_n(s,\ue_{n}(s-), z)) \tilde{N}^{\e^{-1}\varphi}(ds, dz)
\]
and 
\[
J_t = \e^2 \int_0^{t}\int_{\mathbb Z} |g_n(s, \ue_{n}(s-), z)|^2 \tilde{N}^{\e^{-1}\varphi}(ds, dz).
\]
Observe that by Lenglart-Lepingle-Pratelli inequality
\begin{align}
 \mathbb{E}\left(\sup_{0 \le t \le T}|L_{t\wedge \tau_R}|\right)
&\le  
C_3 \mathbb{E} \left[ \left(\int_0^{T \wedge \tau_R}\int_{\mathbb Z} \{2\e (\ue_{n}(s-), g_n(s, \ue_{n}(s-), z))\}^2 \e^{-1}\varphi(s, z) \nu(dz)ds\right)^{1/2}\right]\notag\\
\le & 2\sqrt{\e} C_3 \mathbb{E} [1 + \sup_{0 \le t \le T}|\ue_{n}(s\wedge \tau_R)|^2] 
\left(\sup_{\varphi \in S^M} \int_0^T \int_{\mathbb Z} M_g^2(z)\varphi(s, z)\nu(dz)ds\right)^{1/2}\notag\\
\le & \sqrt{\e} C_4 (1 + \mathbb{E} [\sup_{0 \le t \le T}|\ue_{n}(s\wedge \tau_R)|^2]. \label{mg1}
\end{align}
\begin{align}
\mathbb{E}\left(\sup_{0 \le t \le T} J_{t\wedge \tau_R}\right)
\le & \e^2 \mathbb{E}\left( \int_0^{T \wedge \tau_R}\!\!\int_{\mathbb Z} |g_n(s, \ue_{n}(s-), z)|^2N^{\e^{-1}\varphi}\right)\notag\\
\le  & D  \e \{ 1 + \mathbb{E} [\sup_{0 \le t \le T}|\ue_{n}(s\wedge \tau_R)|^2]\}.
\label{mg2}
\end{align}
Using \eqref{mg1} and \eqref{mg2} in \eqref{gronwall}, and with $\e$ small enough, an upper bound for $\mathbb{E} \left(\sup_{0 \le t \le T\wedge \tau_R} |\ue_{n}(t)|^2\right)$ is obtained which is independent of $R$.  Using it back in  \eqref{zero}    we deduce get that 
\begin{align}\label{eq:preest}
\sup_{0\leq\ep\leq 1}\sup_{\varphi \in \mathcal{U}^M}\mathbb{E} \left\{\sup_{0 \le t \le T\wedge \tau_R} |\ue_{n}(t)|^2 + \int_0^{T \wedge \tau_R} \norm{\ue_{n}(s)}ds\right\} \leq C_1
\end{align}
for some constant $C_1$ independent of $R$. It follows that
\begin{align*}
P\{\tau_R < t\} &= P\{|\ue_{n}(t\wedge \tau_R)|^2 +  \int_0^{t\wedge \tau_R}||\ue_{n}(s)||^2ds \ge R\} \le \frac{K_1}{R} \to 0
\end{align*}
as $R \to \infty$. Hence if $\tau := \displaystyle{\lim_{R \to \infty}} \tau_R$, then $P\{ \tau \ge t \} = 1$  for any $0 \le t \le T$.
The assertion now follows from \eqref{eq:preest} by taking $R\rt \infty$.

%

\end{proof}
\begin{lemma}
For any $\varphi \in \mathcal{U}_M$, where $0 < M < \infty$, 
  the equation 
\begin{equation}
d\ubar_{\varphi}(t) + [\ A\ubar_{\varphi}(t) + B(\ubar_{\varphi}(t), \ubar_{\varphi}(t))] 
= [\f(t) + \int_{\mathbb Z} g(t, \bu_{\varphi}(t), z)(\vphi(t,z) - 1) \nu(dz)]dt,\label{13}
\end{equation}
with $\bu_{\varphi}(0) =
\xi\in H$, 
has a unique strong solution in $\mathrm{L}^2(\Omega; X)$, where
$X=C(0, T; H) \cap \mathrm{L}^2(0, T; V)$. 
 Moreover, there exists a constant $\mathcal{C}_1(T)$ such that
\begin{equation}
 \sup_{\varphi \in \mathcal{S}^M}\lf\{ \sup_{0 \le t \le T}|\ubar_{\varphi}(t)|^2 + \int_0^T\norm{\ubar_{\varphi}(s)}^2\ ds\ri\} \le  \mathcal{C}_1(T).
 \label{sconteqnest1}
\end{equation}

\end{lemma}
The proof is similar to that in Capinski and Gatarek \cite{CG}, and hence omitted.

\begin{theorem}\label{ThmLDP1}
Let $\varphi_n \to \varphi$ in $S^M$ as $n \to \infty$. Then $\ubar^{(n)} \doteq \ubar_{\vphi_n} \rt \ubar_{\vphi}$
in $C([0, T]; H)\cap L^2(0, T: V)$.  
\end{theorem}

\begin{proof}
We first show that the sequence $\{\ubar^{(n)}\}$ is tight in $C([0,T],H)$. To see this note that
\begin{align*}
|\ubar_n(t+h) - \ubar(t)| \leq & \int_{t}^{t+h}\left(|A\ubar_n(s)|+|B(\ubar_n(s), \ubar_n(s))| + |f(s)|\right) ds \\
&+ \lf|\int_{t}^{t+h} g(s, \ubar_n(s), z) (\vphi_n(s,z)-1)\ ds\ri|
\end{align*}	
{\sc Step 1:} By a priori estimates, we know that by Banach-Aloaglu theorem, there exists a $\bv \in L^{\infty}(0, T; H) \cap L^2(0, T; V)$  such that 
along a subsequence $\{n_j\}$, 
\begin{enumerate}
\item[(i)] $\ubar^{(n_j)} \to \bv$ in the weak-star  topology in $L^{\infty}(0, T; H)$,
\item[(ii)] $\ubar^{(n_j)} \to \bv$ weakly  in $L^{2}(0, T; V)$.  
\end{enumerate}
The limit $\bv$ is the one and the same in both topologies because of the imbedding $V \subset_{\to} H$. 

From this, we deduce that 
 $\ubar^{(n_j)} \to \bv$ strongly in $L^2(0, T; H)$ by Theorem \ref{tightness}. 
  Indeed,  each  element of the sequence $\{\ubar^{(n_j)}\}$ belongs to $L^2(0, T; H)\cap C([0, T]; V^{\prime})$ and the sequence is bounded in $L^2(0, T; V)$. In order to to invoke the tightness criterion given in Theorem \ref{tightness},
  we will show that the sequence $\{\ubar^{(n_j)}\}$ is precompact in $C([0, T]; V^{\prime})$ by the Aldous criterion in the course of this proof.  Hence, the set is precompact in $L^2(0, T; H)$, and the strong convergence of a subsequence to $\bv$ in $L^2(0, T: H)$ follows automatically.

It remains to show that $\bv = \ubar_\vphi$. Since the limit is unique, we obtain convergence of the entire sequence. 
 Towards this end, by a slight abuse of notation, we will continue to denote the subsequence $\{n_j\}$ simply by $\{n\}$.
Fix any positive integer $k$, and $ h \in C^1([0, T])$. Without loss of generality, let us take $\norm{h}_{\infty} = 1$. 
Then, for a.e. $t$ in $[0, T]$, an integration by parts yields
\begin{align}
& - \int_0^t (\ubar^{(n)}(s), h^{\prime}(s) e_k)ds +  \int_0^t((\ubar^{(n)}(s), h(s)e_k)) ds +  \int_0^t \langle \bB(\ubar^{(n)}(s), \ubar^{(n)}(s)), h(s) e_k\rangle ds\notag\\
& =  h(0)(\ubar^{(n)}(0),  e_k) - h(t)(\ubar^{(n)}(t), e_k) + \int_0^t \int_{\mathbb Z} (\varphi_{n}(s, z) - 1) (g(s, \ubar^{(n)}(s), z), h(s) e_k)\nu(dz)ds.\label{ktheqn}
\end{align}
By weak-star convergence in $L^{\infty}(0, T; H)$, we obtain that 
\begin{equation}
\int_0^t (\ubar^{(n)}(s), h^{\prime}(s) e_k)ds \to \int_0^t (\bv(s), h^{\prime}(s) e_k)ds.\label{ft}
\end{equation}
Next, by weak convergence of $\ubar^{(n)}$  in $L^2(0, T; V)$, 
\begin{equation}
\lim_{n \to \infty} \int_0^t((\ubar^{(n)}(s), h(s) e_k)) ds = \int_0^t((\bv(s), h(s) e_k)) ds. \label{st}
\end{equation}

\par
To establish  the convergence of the third term on the left side of \eqref{ktheqn}, we define 
$\bw^{(n)} = \ubar^{(n)} - \bv$  and obtain  
\begin{align*}
& |\langle \bB(\ubar^{(n)}(s),\ubar^{(n)}(s)) - \bB(\bv(s), \bv(s)), h(s) e_k\rangle| \\
& \le  C_k(\norm{\bB(\bv(s), \bw^{(n)}(s))}_{_{V\pr}} + \norm{\bB(\bw^{(n)}(s), \bv(s))}_{_{V\pr}} + \norm{\bB(\bw^{(n)}(s), \bw^{(n)}(s))}_{_{V\pr}})\\
& \le C_k [2(|\bv(s)|\,\norm{\bv(s)} \,|\bw^{(n)}(s)| \,\norm{\bw^{(n)}})^{1/2} +  |\bw^{(n)}(s)|\,\norm{\bw^{(n)}(s)}]\\
& \le C_k[\e |\bv(s)|\,\norm{\bv(s)} + (1 + \frac{1}{\e}) |\bw^{(n)}(s)| \,\norm{\bw^{(n)}}]
\end{align*}
for any $\e > 0$. 
Hence, 
\begin{align}
&\int_0^t\,|\langle \bB(\ubar^{(n)}(s), \ubar^{(n)}(s)) - \bB(\bv(s), \bv(s)), h(s)e_k\rangle|ds|\notag\\
& \le C_k \int_0^t\, \left[\e \norm{\bv(s)}^2  + (1 + \frac{1}{\e}) |\bw^{(n)}(s)|\,\norm{\bw^{(n)}(s)}\right]ds
\notag\\
& \le  C_k \left[ \e \norm{\bv}^2_{L^2([0, T]; V)} +  (1 + \frac{1}{\e})\norm{\bw^{(n)}}_{L^2([0, T]; V)}\,\norm{\bw^{(n)}}_{L^2([0, T]; H)}\right].\label{tht}
\end{align}
By  the boundedness of $  \bw^{(n)}$ in $L^2([0, T]; V)$
 uniformly in $n$,  and  strong convergence to zero in $L^2([0, T]; H)$, the second term on the right side of equation \eqref{tht}  converges to $0$ upon letting $n$  to $\infty$. Then let $\e$ to $0$ to conclude that
\[ \lim_{n \to \infty} \int_0^t\,|\langle \bB(\ubar^{(n)}(s), \ubar^{(n)}(s)) - \bB(\bv(s), \bv(s)), h(s)e_k\rangle|ds| = 0.
\]
Consider the first two terms on the right side of equation \eqref{ktheqn}.
The convergence of $h(0)(\ubar^{(n)}(0),  e_k) - h(t)(\ubar^{(n)}(t), e_k) $ to $h(0)(\xi,  e_k) - h(t)(\bv(t), e_k)$
as $n \rt \infty$ follows easily for a.e. $t \in [0, T]$ by the convergence of the subsequence $\{\ubar^{(n)}\}$ 
in $D([0, T]; V\pr)$. 

{\sc Step 2:} Denote the integral on the right side of \eqref{ktheqn} by $J(\ubar^{(n)}, \varphi_{n})$, and we show that along a subsequence 
$J(\ubar^{(n)}, \varphi_{n}) \to J(\bv, \varphi)$ a.s. 
First, we will show that along a subsequence 
\begin{equation}
\lim_{n \to \infty} [J(\ubar^{(n)}, \varphi_{n})(t) - J(\ubar^{(n)}, \varphi)(t)] = 0\,\,\text{a.s}. \label{firstpart}
\end{equation}
Indeed by  \eqref{h2}, given any $\e > 0$, there exists a compact set $K_{\e} \subset Z$ such that 
\begin{align}
&|\int_0^t \int_{K_{\e}^c} (\varphi_{n}(s, z) - 1) (g(s, \ubar^{(n)}(s), z), h(s)e_k)\nu(dz)ds| \notag\\
& \le \int_0^T \int_{K_{\e}^c} |\varphi_{n}(s, z) - 1| \,|g(s, \ubar^{(n)}(s), z)|\nu(dz)ds\notag\\
& \le \int_0^T \int_{K_{\e}^c} |\varphi_{n}(s, z) - 1| \,M_g(z) (1 + |\ubar^{(n)}(s)|)\nu(dz)ds\notag\\
& \le (1 + \sup_{0 \le s \le T} |\ubar^{(n)}(s)|)\,\e. \label{fp1}
\end{align}
Likewise, 
\begin{equation}
|\int_0^t \int_{K_{\e}^c} (\varphi(s, z) - 1) (g(s, \ubar^{(n)}(s), z), h(s) e_k)\nu(dz)ds| 
\le (1 + \sup_{0 \le s \le T} |\ubar^{(n)}(s)|)\,\e. \label{fp1.5}
\end{equation}
Next, for any fixed $R > 0$, consider 
\begin{align}
& |\int_0^t \int_{K_{\e}} g(s, \ubar^{(n)}(s), z) \varphi_{n}(s, z) 1_{\{M_g(z) \ge R\}} \nu(dz)ds|\notag\\
& \le (1 + \sup_{0 \le s \le T} |\ubar^{(n)}(s)|) \int_0^T \int_{K_{\e}}M_g(z) \,\varphi_{n}(s, z) \,1_{\{M_g(z) \ge R\}} \nu(dz)ds\notag\\
& \le (1 + \sup_{0 \le s \le T} |\ubar^{(n)}(s)|)\,\e \,\,\,\,\text{if $R$ is large enough by invoking \eqref{h3};}\nonumber\\
& \le C\e. \label{fp2}
\end{align} 
Likewise,
\begin{equation}
|\int_0^t \int_{K_{\e}} g(s, \ubar^{(n)}(s), z) \varphi(s, z) 1_{\{M_g(z) \ge R\}} \nu(dz)ds| \le  C\,\e.\label{fp3}
\end{equation}
Finally, on $[0, T]\times Z$ consider the function $(g(s, \ubar^{(n)}(s), z), e_k) 1_{\{M_g(z) < R\}}$.
Clearly, 
\[
|(g(s, \ubar^{(n)}(s), z), e_k)| 1_{\{M_g(z) < R\}} \le C\,R.
\]
Also, $|(g(s, \bv, z), e_k)| 1_{\{M_g(z) < R\}} \le C\,R$. 
By local compactness of $Z$, there exists a function $\beta$ in $C_c(Z; [0, 1])$ such that $\beta = 1$ on $K_{\e}$. Define $d\mu_1 := \beta(z) \nu(dz)ds$ on $Z \times [0, T]$, and the normalized measure $\mu := \mu_1/(\mu_1(Z \times [0,T]))$. Let $K$ denote the support of $\beta$ and $\kappa$, the normalizing constant. 
 Then 
\begin{align*}
 \int_0^T \int_{Z} |g(s, \bu_{j\pr}(s), z) - g(s, \bv(s), z)| d\mu
 &\le  \int_0^T \int_{Z} L_g(z) |\bu_{j\pr}(s) - \bv(s)|  d\mu. \\
& \le \frac{1}{\kappa} L_1\sqrt{T}\,\norm{\bu_{j\pr} - \bv}_{L^2([0, T] ; H)}\end{align*}
which $\to 0$ as $j\pr \to \infty$. Hence, there exists a subsequence $\{i\pr\}$ of $\{j\pr\}$ such that 
\[
\lim_{i\pr \to \infty} (g(s, \bu_{i\pr}(s), z), h(s) e_k)1_{\{M_g(z) < R\}} =  (g(s, \bv(s), z), h(s) e_k)1_{\{M_g(z) < R\}} \,\,\,\,\mu\,\,\text{a.s.}
\] 
 Hence, 
\[
\lim_{i\pr \to \infty} (g(s, \bu_{i\pr}(s), z), h(s) e_k)1_{\{M_g(z) < R\}}1_{K_{\e}}(z) =  (g(s, \bv(s), z), h(s) e_k)1_{\{M_g(z) < R\}}1_{K_{\e}}(z) \,\,\,\,\mu\,\,\text{a.s.}
\] 
 Let us  define the measure $\eta_{i\pr}(A) := \int_A \varphi_{i\pr}(s, z)\beta(z) \nu(dz)ds$ for all Borel sets 
$A \subseteq [0, T]\times Z$.  Let $\rho_{i\pr}$ denote the corresponding normalized measure, 
that is, 
\[
\rho_{i\pr} = \frac{\eta_{i\pr}}{m_{i\pr}}
\]
where $m_{i\pr} = \int_0^T \int_{\mathbb Z} \varphi_{i\pr}(s, z)\beta(z) \nu(dz)ds$. 

Likewise, define $\rho$ and $m$ when $\varphi$ is used  instead of $\varphi_{i\pr}$.   Then, denoting 
the relative entropy  of $\rho_{i\pr}$ with respect to $\mu$ by $R(\rho_{i\pr} || \mu)$, we have
\[
\sup_{i\pr} R(\rho_{i\pr} || \mu) < \infty
\]
since $m_{i\pr} \to m$ by using the fact that  $\mathcal{M}_{FC}(Z)$ is endowed with vague topology (see Budhiraja, Chen and Dupuis \cite{BCD}). Also, $\rho_{i\pr} \to \rho$ weakly since $\beta$ is in $C_c(Z; [0, 1])$ and $m_{i\pr} \to m$. Hence, by Lemma 2.8 in Bou\'e and Dupuis \cite{BD},
we conclude that
\begin{align}
&\lim_{i\pr \to \infty}\frac{1}{m_{i\pr}}\int_0^t\int_{\mathbb Z}  (g(s, \bu_{i\pr}(s), z), h(s)e_k)1_{\{M_g(z) < R\}}1_{K_{\e}}(z)\varphi_{i\pr}(s, z)\beta(z)\nu(dz)ds \notag\\
&= 
 \frac{1}{m}\int_0^t\int_{\mathbb Z}  (g(s, \bv(s), z), h(s)e_k)1_{\{M_g(z) < R\}}1_{K_{\e}}(z)\varphi(s, z)\beta(z)\nu(dz)ds.\label{fp4}
\end{align}
Analogously, 
\begin{align}
&\lim_{i\pr \to \infty}\frac{1}{m_{i\pr}}\int_0^t\int_{K_{\e}} (g(s, \bu_{i\pr}(s), z), h(s)e_k)1_{\{M_g(z) < R\}}
\varphi(s, z)\nu(dz)ds \notag\\
&= \frac{1}{m}\int_0^t\int_{K_{\e}} (g(s, \bv(s), z), h(s)e_k)1_{\{M_g(z) < R\}}\varphi(s, z)\nu(dz)ds.\label{fp5}
\end{align}
Hence, by \eqref{fp4} and \eqref{fp5}, given any $\d > 0$, there exists a large $N$ such that if $i\pr \ge N$, then 
\begin{equation}
|\int_0^t\int_{K_{\e}} (g(s, \bu_{i\pr}(s), z), e_k)1_{\{M_g(z) < R\}}\left(\varphi_{i\pr}(s, z) - \varphi(s, z)\right) \nu(dz)ds| < \d.\label{fp}
\end{equation}
By using the estimates \eqref{fp1}, \eqref{fp1.5}, \eqref{fp2}, \eqref{fp3} and \eqref{fp}, the proof of \eqref{firstpart} is over.

\par
Next, we need to show that 
\begin{equation}
\lim_{i\pr \to \infty} [J(\bu_{i\pr}, \varphi)(t) - J(\bv, \varphi)(t)] = 0.\label{secondpart}
\end{equation}
Towards this, fix any $\d  > 0$ and define the set 
\[
A_{i\pr} : = \{t \in [0, T]: |\bu_{i\pr}(t) - \bv(t)| > \d\}.
\]
Then, with $\l$ denoting  Lebesgue measure,
\[
\l(A_{i\pr})  \le  \frac{1}{\d^2}\int_0^T |\bu_{i\pr}(t) - \bv(t)|^2 dt 
\]
which converges to $0$ as $i\pr \to \infty$.  Consider 
\begin{align}
&| J(\bu_{i\pr}, \varphi)(t) - J(\bv, \varphi)(t)|\notag\\
&\le \int_0^T \int_{\mathbb Z} |(g(s, \bu_{i\pr}(s), z) - g(s, \bv(s), z)| \,| |\varphi(s, z) - 1|\nu(dz)ds\notag\\
& \le \int_0^T \int_{\mathbb Z} L_g(z) |\bu_{i\pr}(s) - \bv(s)|\, |\varphi(s, z) - 1|\nu(dz)ds\notag\\
& \le 2K \int_{A_{i\pr}}\int_{\mathbb Z} L_g(z)\,|\varphi(s, z) - 1|\nu(dz)ds + \d \int_{A^c_{i\pr}}\int_{\mathbb Z} L_g(z)\,|\varphi(s, z) - 1|\nu(dz)ds\notag\\
& \le  2K \int_{A_{i\pr}}\int_{\mathbb Z} L_g(z)\,|\varphi(s, z) - 1|\nu(dz)ds +
\d \int_0^T\int_{\mathbb Z} L_g(z)\,|\varphi(s, z) - 1|\nu(dz)ds \label{timesplit} 
\end{align}
with $K := \sup_{i\pr}  |\bu_{i\pr}|_{L^{\infty}([0, T]; H)} \vee |\bv|_{L^{\infty}([0, T]; H)}$.
In \eqref{timesplit},  use \eqref{mlcond2} and \eqref{mlcond3} while 
allowing  $i\pr \to \infty$, and then letting $\d \to 0$ to conclude that
\[
| J(\bu_{i\pr}, \varphi)(t) - J(\bv, \varphi)(t)| \to 0.
\]
 From \eqref{firstpart} and \eqref{secondpart}, we obtain 
\begin{equation}
\int_0^T \int_{\mathbb Z} (\varphi_{i\pr}(s, z) - 1) (g(s, \bu_{i\pr}(s), z), h(s)e_k)\nu(dz)ds \to \int_0^T \int_{\mathbb Z} (\varphi(s, z) - 1) (g(s, \bv(s), z), h(s)e_k)\nu(dz)ds
\end{equation}
a.s. as $i\pr \to \infty$. This along with \eqref{ft}, \eqref{st}, and \eqref{tht} proves that 
$\bv = \bu$ on $[0, T]$ where $\bu$ is shorthand for $\bu_{\varphi}$. 

\sni
{\sc Step 3:} With $\bw_{i\pr} : = \bu_{i\pr} - \bu$, the energy equality yields
we have 
\begin{align}
&|\bw_{i\pr}(t)|^2 + 2\nu \int_0^t \norm{\bw_{i\pr}(s)}^2ds  + 2\int_0^t \langle \bB(\bu_{i\pr}(s), \bu_{i\pr}(s)) - \bB(\bu(s), \bu(s)), \bw_{i\pr}(s)\rangle ds \notag\\
& =  2\int_0^t \int_{\mathbb Z} (g(s, \bu_{i\pr}(s), z)(\varphi_{i\pr}(s, z) - 1) - g(s, \bu(s), z)(\varphi(s, z) - 1), \bw_{i\pr}(s))\nu(dz)ds.\label{weqn}
\end{align}
Consider the expression
\begin{align}
&|\langle \bB(\bu_{i\pr}(s), \bu_{i\pr}(
s)) - \bB(\bu(s), \bu(s)), \bw_{i\pr}(s)\rangle|\notag\\
& = | \langle \bB(\bu(s), \bw_{i\pr}(s)) + \bB(\bw_{i\pr}(s), \bu(s))  + \bB(\bw_{i\pr}(s), \bw_{i\pr}(s)), \bw_{i\pr}(s)\rangle |\notag\\
& = |  \langle \bB(\bw_{i\pr}(s), \bu(s)), \bw_{i\pr}(s) \rangle |\notag\\
& \le C |\bw_{i\pr}(s)| \,\norm{\bw_{i\pr}(s)}\,\norm{\bu(s)}\notag\\
& \le \frac{1}{2} \norm{\bw_{i\pr}(s)}^2 + \frac{C^2}{2}|\bw_{i\pr}(s)|^2 \,\norm{\bu(s)}^2.\label{btforw}
\end{align}
Next, we write 
\[
\int_0^t \int_{\mathbb Z} (g(s, \bu_{i\pr}(s), z)(\varphi_{i\pr}(s, z) - 1) - g(s, \bu(s), z)(\varphi(s, z) - 1), \bw_{i\pr}(s))\nu(dz)ds = I_1 + I_2,\,\,\,\text{where}
\]
\begin{align*}
I_1 &= \int_0^t \int_{\mathbb Z} (g(s, \bu_{i\pr}(s), z)(\varphi_{i\pr}(s, z) - 1) - g(s, \bu_{i\pr}(s), z)(\varphi(s, z) - 1), \bw_{i\pr}(s))\nu(dz)ds,\,\,\text{and}\\
I_2 &= \int_0^t \int_{\mathbb Z} (\{g(s, \bu_{i\pr}(s), z) - g(s, \bu(s), z)\}(\varphi(s, z) - 1), \bw_{i\pr}(s))\nu(dz)ds
\end{align*}
\begin{align*}
|I_1|  \le &\int_0^t \int_{\mathbb Z} | (g(s, \bu_{i\pr}(s), z)(\varphi_{i\pr}(s, z) - 1), \bw_{i\pr}(s))|\nu(dz)ds\\ &+ \int_0^t \int_{\mathbb Z} | (g(s, \bu_{i\pr}(s), z)(\varphi(s, z) - 1), \bw_{i\pr}(s))|\nu(dz)ds,
\end{align*}
and we denote the right side of the above inequality by $J_1 + J_2$. Then,
\begin{align}
J_1 & \le \int_0^T \int_{\mathbb Z} |g(s, \bu_{i\pr}(s), z)| \,|(\varphi_{i\pr}(s, z) - 1)| \,|\bw_{i\pr}(s))|\nu(dz)ds\notag\\
& \le \int_0^T \int_{\mathbb Z} M_g(z)(1 + |\bu_{i\pr}(s)|) \,|(\varphi_{i\pr}(s, z) - 1)| \,|\bw_{i\pr}(s))|\nu(dz)ds\notag\\
& \le  \sup_{0 \le s \le T} (1 + |\bu_{i\pr}(s)|)  \int_0^T \int_{\mathbb Z} M_g(z) |(\varphi_{i\pr}(s, z) - 1)| \,
|\bw_{i\pr}(s))|\nu(dz)ds\notag\\
&\le  K \int_0^T \int_{\mathbb Z} M_g(z) |(\varphi_{i\pr}(s, z) - 1)| \,
|\bw_{i\pr}(s))|\nu(dz)ds\label{j1forw}
\end{align}
which $ \to 0$ as ${i\pr} \to \infty$ by using the method employed that led to the convergence of  \eqref{timesplit} to zero.
Likewise,  $J_2 \to 0$ as $i\pr \to \infty$. Hence, given any $\d > 0$, there exists 
a large enough $N$ such that if $i\pr \ge N$, then $I_1 \le \d$.

Next, consider 
\begin{align}
I_2& \le \int_0^t \int_{\mathbb Z} |g(s, \bu_{i\pr}(s), z) - g(s, \bu(s), z)|\,|(\varphi(s, z) - 1)|\, |\bw_{i\pr}(s))|\nu(dz)ds\notag\\
& \le \int_0^t \int_{\mathbb Z} L_g(z) \,|(\varphi(s, z) - 1)|\, |\bw_{i\pr}(s))|^2\nu(dz)ds\label{itwo}
\end{align}
Using \eqref{btforw}, \eqref{j1forw}, \eqref{itwo} in \eqref{weqn}, we obtain for all 
$i\pr \ge N$, 
\begin{align}
&|\bw_{i\pr}(t)|^2 +  \int_0^t \norm{\bw_{i\pr}(s)}^2ds \notag\\
&\le C^2 \int_0^t |\bw_{i\pr}(s)|^2 \,\norm{\bu(s)}^2 ds + \int_0^t \int_{\mathbb Z} L_g(z) \,|(\varphi(s, z) - 1)|\, |\bw_{i\pr}(s))|^2\nu(dz)ds + \d. \label{ineq4w}
\end{align}
Ignoring the second term on the left side, and using Gronwall inequality, 
we obtain 
\[
|\bw_{i\pr}(t)|^2 \le K_1(T)\d,
\]
where $K_1(T)$ is a finite constant by \eqref{mlcond2}, which holds uniformly for all $i\pr \ge N$. 
The arbitrariness of $\d$ yields
\[
\lim_{i\pr \to \infty} \sup_{0 \le t \le T} |\bw_{i\pr}(t)|^2 = 0.
\]
Using this in \eqref{ineq4w}, the proof is over.
\end{proof}
\begin{theorem}
Let hypotheses ${\bf{H.1}}$ - ${\bf{H.8}}$ hold. If $\varphi_{\e}, \varphi \in {\mathcal{U}}^{M}$ such that $\varphi_{\e} \implies \varphi$ as $\e \to 0$, then $
\mathcal{G}^{\e}(\e N^{\e^{-1}\varphi_{\e}}) \implies \mathcal{G}^0(\varphi) 
$
as $\mathbb{D}([0, T]; V^{\prime})\cap L^2([0, T]; V)$ valued random elements.
\end{theorem}
\begin{proof}
Let $\ue_{\varphi_{\e}}$  be the solution of equation \eqref{11} with $\varphi_{\e}$ in the place of $\varphi$. We will denote 
 $\ue_{\varphi_{\e}}$ by $\mathcal{G}^{\e}(\e N^{\e^{-1}\varphi_{\e}})$. We know that the laws of the family
 $\{\ue_{\varphi_{\e}}\}$ indexed by $\e > 0$ for $\e$ small enough, form a tight family as random elements taking values in $\mathbb{U} := L^2([0, T]; H) \cap D([0, T]; V^{\prime})$ by using Theorem 3.1. For this, precompactness in $D([0, T]; V^{\prime})$ needs to be shown. We prove it below by  the Aldous criterion.

Let $\e_0$ be small enough so that there exists a unique strong solution $\ue_{\varphi_{\e}}$ for each $\e \in (0, \e_0]$. First, we verify that  $\{\ue_{\varphi_{\e}}(t): 0 < \e \le \e_0\}$ is a tight 
family of $V^{\prime}$-valued random variables for each $t \in [0, T]$. From the a priori bound it follows in particular that
$\sup_{0 < \e \le \e_0} \mathbb{E}|\ue_{\varphi_{\e}}(t)|^2 < \infty$. For each  $v \in V^{\prime}$, the real-valued family $\{(\ue_{\varphi_{\e}}(t), v)_{V^{\prime}}\}$ is therefore tight. 

Fix a $k \in \mathbb{N}$. Next, consider the first $k$ number of eigenfunctions $\{e_1, e_2,\cdots,e_k\}$ of the Stokes operator $\bA$ which form a finite orthonormal sequence in $H$ with eigenvalues $\l_1, \cdots, \l_k$ respectively. Hence, $F_k := \{\sqrt{\l_j}e_j: j = 1, 2, \cdots, k\}$ is an orthonormal 
set in $V^{\prime}$. The subspace spanned by $F_k$ is denoted $S_k$. 
Let $d(\ue_{\varphi_{\e}}(t), S_k)$ denote the distance in $V\pr$ of $\ue_{\varphi_{\e}}(t)$ from $S_k$. Then, for any $\delta > 0$, 
\begin{align*}
 \limsup_{\e \to 0} P\{d(\ue_{\varphi_{\e}}(t), S_k) > \delta\} & \le 
\limsup_{\e \to 0} \d^{-2} \mathbb{E}\left(d(\ue_{\varphi_{\e}}(t), S_k)^2\right)\\
& = \d^{-2} \limsup_{\e \to 0} \mathbb{E} \sum_{j = k+1}^{\infty} (\ue_{\varphi_{\e}}(t), \sqrt{\l_j}e_j)^2_{V^{\prime}}\\
& = \d^{-2} \limsup_{\e \to 0} \mathbb{E} \sum_{j = k+1}^{\infty} \frac{1}{\l_j}(\ue_{\varphi_{\e}}(t), e_j)^2\\
&\le \frac{\d^{-2}}{\l_{k + 1}} \,\sup_{0 < \e \le \e_0} \mathbb{E} |\ue_{\varphi_{\e}}(t)|^2
\end{align*}
which tends to zero as $k \to \infty$. This holds for any fixed $t \in [0, T]$. 
Hence we infer that $\{\ue_{\varphi_{\e}}(t): 0 < \e \le \e_0\}$ is a tight 
family of $V^{\prime}$-valued random variables for each $t \in [0, T]$ (cf.  Metivier (see \cite{Met} pages 28-29). 

Finally, we need to verify  Aldous condition: For each $\rho > 0, \eta > 0$, there exists a $\d > 0$ and $\e_0 > 0$ with the property that whatever be the family of stopping times $\{\tau_{\e}\}$ with $\tau_{\e} + \delta \le T$ for all $\e$, 
the following holds:
$
 \sup_{\e \le \e_0} \,\sup_{0 \le \theta \le \d} P\{d(\ue_{\varphi_{\e}}(\tau_{\e}), \ue_{\varphi_{\e}}(\tau_{\e} + \theta)) \ge \eta\} \le \rho.
$ We first write 
\begin{align}
&\ue_{\varphi_{\e}}(\tau_{\e} + \theta) -  \ue_{\varphi_{\e}}(\tau_{\e}) + \int_{\tau_{\e}}^{\tau_{\e} + \theta} [\bA\ue_{\varphi_{\e}}(s) + \bB(\ue_{\varphi_{\e}}(s))] ds \notag\\
&\!\!\!=\!\!\int_{\tau_{\e}}^{\tau_{\e} + \theta}\!\!\!\!\f(s)ds + \int_{\tau_{\e}}^{\tau_{\e} + \theta}\!\!\!\!\int_{\mathbb Z} g(s, \ue_{\varphi_{\e}}(s), z)(\varphi_{\e}(s, z) - 1) \nu(dz)ds  + {\e}\int_{\tau_{\e}}^{\tau_{\e} + \theta}\!\!\!\!\int_{\mathbb Z} g(s, \ue_{\varphi_{\e}}(s), z) \tilde{N}^{\e^{-1}\varphi_{\e}}(ds, dz)\notag\\
&=: J_1 + J_2 + J_3.
\end{align}
For any $v \in V$, consider  
\begin{align}
  \mathbb{E}| \langle \int_{\tau_{\e}}^{\tau_{\e} + \theta} [\bA\ue_{\varphi_{\e}}(s) + \bB(\ue_{\varphi_{\e}}(s))] ds, v\rangle| & \le |v|\, \mathbb{E} \int_{\tau_{\e}}^{\tau_{\e} + \theta} \{\norm{\ue_{\varphi_{\e}}(s)} + |\ue_{\varphi_{\e}}(s)|\;\norm{\ue_{\varphi_{\e}}(s)}\}ds\notag\\
& \le |v| \,[\mathbb{E} \sup_{0 \le s \le T} (1 + |\ue_{\varphi_{\e}}(s)|)^2]^{1/2}\, [\mathbb{E} (\int_{\tau_{\e}}^{\tau_{\e} + \theta} \norm{\ue_{\varphi_{\e}}(s)}ds)^2]^{1/2}\notag\\
& \le C_1|v| \sqrt{\theta} [\mathbb{E}\int_{\tau_{\e}}^{\tau_{\e} + \theta}\norm{\ue_{\varphi_{\e}}(s)}^2ds]^{1/2}.
\label{left}
\end{align}
Thus,
$\mathbb{E}|\int_{\tau_{\e}}^{\tau_{\e} + \theta} [\bA\ue_{\varphi_{\e}}(s) + \bB(\ue_{\varphi_{\e}}(s))] ds|_{V^{\prime}} \le K_1 \sqrt{\theta}$ for a suitable constant $K_1$ which is independent of $\e$ by the a priori estimate on $\bu_{\varphi_{\e}}^{\e}$. It is clear that 
$|J_1|_{V^{\pr}} \le K_2 \sqrt{\theta}$ for a suitable constant $K_2$.
Next, consider
\begin{align}
 \mathbb{E} |J_2|_{V^{\pr}}& \le \mathbb{E} \left(\sup_{0 \le s \le T} (1 + |\ue_{\varphi_{\e}}(s)|)\int_{\tau_{\e}}^{\tau_{\e} + \theta}\int_{\mathbb Z} M_g(z) |\varphi_{\e}(s, z) - 1|\nu(dz)ds\right)\notag\\
& \le \mathbb{E} \left(\sup_{0 \le s \le T} (1 + |\ue_{\varphi_{\e}}(s)|)\sup_{\phi \in S^M} \int_{\tau_{\e}}^{\tau_{\e} + \theta}\int_{\mathbb Z} M_g(z)\,|\phi(s, z) - 1|\nu(dz)ds\right) 	\label{right2}
\end{align}
which by (iii) of Lemma C.3 converges to zero as $\theta \to 0$. 
Finally, for any $v \in V$, 
\begin{align}
 \mathbb{E} |\langle J_3, v\rangle| & \le |v| \sqrt{\e} \mathbb{E} (\int_{\tau_{\e}}^{\tau_{\e} + \theta} \int_{\mathbb Z} |g(s, \ue_{\varphi_{\e}}(s), z)|^2\varphi_{\e}(s, z)\nu(dz) ds)^{1/2}\notag\\
&\le |v| \sqrt{\e} \mathbb{E} \sup_{0 \le s \le T} (1 + |\ue_{\varphi_{\e}}(s)|) \sup_{\phi \in S^M}(\int_{0}^{T} \int_{\mathbb Z} 
M_g^2(z) \phi(s, z)\nu(dz) ds)^{1/2} \notag
\end{align}
so that 
\begin{equation}
\mathbb{E}\norm{J_3}_{_{V\pr}} \le K_3\sqrt{\e} \label{right3}
\end{equation}
which goes to zero as $\e \to 0$. From \eqref{left}, \eqref{right2} and \eqref{right3}, the Aldous condition holds, and we have tightness of the laws of $\{\ue_{\varphi_{\e}}\}$  in $\mathbb{D}([0, T]; V^{\prime})$.

Let $\bu$ be any limit point of 
 $\{\ue_{\varphi_{\e}}\}$. We need to identify the law of $\bu$ with  $\mathcal{G}^0(\varphi)$. 

Towards this, define the process $\bv^{\e}$ as the solution of 
\begin{equation}
 \bv^{\e}(t) + \int_0^t \bA\bv^{\e}(s)ds = \e \int_0^t \int_{\mathbb Z} g(s, \bu_{\varphi_{\e}}^{\e}(s), z)\tilde{N}^{\e^{-1}\varphi_{\e}}(ds, dz)
\label{vepeqn}
\end{equation}
on $[0, T]$. There exists a unique solution of \eqref{vepeqn} with values in $\mathbb{U}$. It is easy to show that 
\begin{equation}
 \lim_{\e \to 0} E\{\sup_{0 \le t \le T} |\bv^{\e}(t)|^2 + \int_0^T \norm{\bv^{\e}(t)}^2dt\} = 0. \label{vepest}
\end{equation}

The family $({\varphi}_{\e}, \e {N}^{\e^{-1}\varphi_{\e}}, 
 {\ue}_{\varphi_{\e}}, {\bv}^{\e}) $ indexed by $\e > 0$ is tight in $\mathcal{U}^{\e}\times \bar{\mathbb{M}}
 \times \mathbb{U}\times \mathbb{U}$.
 By Skorohod representation theorem, we will, by a change of the original probability space to a new 
 space, $(\O\pr, {\mathcal{F}\pr}, {\mathcal{F}_t\pr}, P\pr)$, construct  random elements 
 $(\hat{\varphi}_{\e}, \e \hat{N}^{\e^{-1}\varphi_{\e}}, \hat{\bar{\bu}}^\e_{\varphi_{\e}}, \hat{\bv}^{\e})$ such that 
 
 (i) law of $(\hat{\varphi}_{\e}, \e \hat{N}^{\e^{-1}\varphi_{\e}}, \hat{\bar{\bu}}^\e_{\varphi_{\e}}, \hat{\bv}^{\e}) = $  law of $({\varphi}_{\e}, \e {N}^{\e^{-1}\varphi_{\e}}, 
 {\ue}_{\varphi_{\e}}, {\bv}^{\e}) $ for each $\e > 0$, 
 
 (ii) law of $(\hat{\varphi}, \hat{\bu}) =$  law of $({\varphi}, {\bu})$ and 
 
 (iii) $(\hat{\varphi}_{\e}, \e \hat{N}^{\e^{-1}\varphi_{\e}}, \hat{\bar{\bu}}^\e_{\varphi_{\e}}, \hat{\bv}^{\e})  \to (\hat{\varphi}, 0, \hat{\bu}, 0) $ a.s. as $\e \to 0$ 
along a subsequence. 
 
 \par
Consider $\bw^{\e}(t):= \hat{\ue}_{\varphi_{\e}}(t) - \hat{\bv}^{\e}(t)$. 
Then 
\[
 \bw^{\e}(t) + \int_0^t \bA\bw^{\e}(s)ds + \int_0^t \bB(\hat{\bar{\bu}}^\e_{\varphi_{\e}}(s))ds = \bu(0) + \int_0^t \int_{\mathbb Z} (\hat{\varphi}_{\e}(s, z) - 1) g(s, \hat{\bar{\bu}}^\e_{\varphi_{\e}}(s), z) \nu(dz)ds.
\]
We rewrite the above equation as 
\[
 \bw^{\e}(t) + \int_0^t \bA\bw^{\e}(s)ds + \int_0^t \bB(\bw^{\e}(s) + \hat{\bv}^{\e}(s))ds = \bu(0) + \int_0^t \int_{\mathbb Z} (\hat{\varphi}_{\e}(s, z) - 1) g(s, \bw^{\e}(s) + \hat{\bv}^{\e}(s), z) \nu(dz)ds.
\]
Define the process $\hat{\bu}$ on $[0, T]$ by
\[
\hat{\bu}(t) + \int_0^t \bA\hat{\bu}(s)ds + \int_0^t \bB(\hat{\bu}(s))ds =  \hat{\bu}(0) + \int_0^t \int_{\mathbb Z} (\hat{\varphi}(s, z) - 1) g(s, \hat{\bu}(s), z) \nu(dz)ds.
\]
A proof along lines similar to that of Theorem \ref{ThmLDP1} yields 
\[
 \lim_{\e \to 0} \sup_{0 \le t \le T} |\bw^{\e}(t) - \hat{\bu}(t)|^2 = 0.
\]
Hence $\displaystyle{\lim_{\e \to 0} \sup_{0 \le t \le T} |\hat{\bar{\bu}}^\e_{\varphi_{\e}}(t) - \hat{\bu}(t)| }= 0$  by writing
\[
 |\hat{\bar{\bu}}^\e_{\varphi_{\e}}(t) - \hat{\bu}(t)| \le |\bw^{\e}(t) - \hat{\bu}(t)| 
+ |\hat{\bv}^{\e}(t)|
\]
taking supremum over $[0, T]$ and using \eqref{vepest} which gives uniform convergence of $\hat{\bv}^{\e}$ to $0$. Using it along with 
the a priori estimates, the proof is completed.
 \end{proof}

\section{Proof of MDP}

Our aim in this section is to prove a moderate deviations result for $Y^{\e}$ which corresponds to 
a LDP result with speed $b(\e)$. 
 
\par
Recalling the definition of $\mathcal{U}_{+, \e}^M$ given in Appendix D, we note that for every 
$\varphi \in \mathcal{U}_{+, \e}^M$, there exists a unique solution of the equation 
\begin{align}
&d\ue_{\varphi}(t) + [\bA\ue_{\varphi}(t) + \bB(\ue_{\varphi}(t), \ue_{\varphi}(t))] dt \notag\\
&= [\f(t) + \int_{\mathbb Z} g(t, \ue_{\varphi}(t), z)(\varphi(t, z) - 1) \nu(dz)]dt  + {\e}\int_{\mathbb Z} g(t, \ue_{\varphi}(t), z) \tilde{N}^{\e^{-1}\varphi}(dt, dz).\label{ue}
\end{align}
For, if we set $\tilde{\varphi} = \frac{1}{\varphi}$, by a theorem of Jacod and Shiryayev (Theorem III. 3. 24)
and Lemma 2.3 in \cite{BDM09}, it follows that 
\[
\mathcal{E}_t^{\e}(\tilde{\varphi}) = \exp\left\{\int_{(0, t]\times Z \times [0, \e^{-1}\varphi]}\log(\tilde{\varphi} )d{N} + \int_{(0, t]\times Z \times [0, \e^{-1}\varphi]}(1 - \tilde{\varphi})d\nu_Tdr\right\}
\]
is a martingale. Hence, 
\[
Q_T^{\e}(G) = \int_G \mathcal{E}_T^{\e}(\tilde{\varphi}) d\bar{P}
\]
for $G \in \mathcal{B}(\bar{M})$ defines a probability measure on $\bar{M}$. Further, $\bar{P}$ and 
$Q_T^{\e}$ are mutually absolutely continuous. The random measure $\e N^{\e^{-1}\varphi}$ under $Q_T^{\e}$ has the same law as that of $\e N^{\e^{-1}}$ under $\bar{P}$. Therefore, 
if we denote $\ue$ by $\bar{\mathcal{G}}^{\e}(\e N^{\e^{-1}})$,then 
$\ue_{\varphi} = \bar{\mathcal{G}}^{\e}(\e N^{\e^{-1}\varphi})$ is $Q_T^{\e}$ a.s. the unique solution of 
\eqref{ue}. 

\par
We will write $Y^{\e} := \mathcal{H}^{\e}(\e N^{\e^{-1}})$.
Define $Y^{\e, \varphi} = \mathcal{H}^{\e}(\e N^{\e^{-1}\varphi})$ and hence, 
 \begin{equation}
 Y^{\e, \varphi} =  \frac{\ue_{\varphi}(t) - \bu(t)}{a(\e)} \label{yphiepsilon}
 \end{equation}
  under $Q_T^{\e}$ has the same law as $Y^{\e}$ under $P$. 
 The process $Y^{\e, \varphi}$ solves the equation 
\begin{align}
& dY^{\e, \varphi}(t) + \bA Y^{\e, \varphi}(t)dt + \frac{1}{a(\e)}\left[\bB(\bu^{\e, \varphi}(t)) - \bB(\bu(t))\right]dt \notag\\
&= \frac{\e}{a(\e)} \int_{\mathbb Z} g(t, \bu_{\varphi}^{\e}(t), z) \tilde{N}^{\e^{-1}\varphi}(dz, dt) + \frac{1}{a(\e)} \int_{\mathbb Z} g(t, \bu_{\varphi}^{\e}(t), z) (\varphi(t, z) - 1) \nu(dz)dt\label{yphieps}
\end{align}
Using  
\[
 \frac{1}{a(\e)}\left[\bB(\bu_{\varphi}^{\e}(t)) - \bB(\bu(t))\right] = a(\e)\bB(Y^{\e, \varphi}(t)) + \bB(Y^{\e, \varphi}(t), \bu(t)) + \bB(\bu(t), Y^{\e, \varphi}(t)), 
\]
and the notation $\psi = \frac{\varphi - 1}{a(\e)}$, 
we write  \eqref{yphieps} as   
\begin{align}
& dY^{\e, \varphi}(t) + \bA Y^{\e, \varphi}(t)dt + \left[a(\e)\bB(Y^{\e, \varphi}(t)) + \bB(Y^{\e, \varphi}(t), \bu(t)) + \bB(\bu(t), Y^{\e, \varphi}(t))\right]dt \notag\\
&= \frac{\e}{a(\e)} \int_{\mathbb Z} g(t, \bu_{\varphi}^{\e}(t), z) \tilde{N}^{\e^{-1}\varphi}(dz, dt) \notag\\
& \,\, +  \int_{\mathbb Z} \{g(t, \bu_{\varphi}^{\e}(t), z) - g(\bu(t), z)\} \psi(z, t) \nu(dz)dt
 + \int_{\mathbb Z} g(t, \bu(t), z) \psi(z, t) \nu(dz)dt. \label{yepsphieqn}
\end{align}
Next, as $\e$ varies, we choose a corresponding control $\varphi_{\e}$, and using the notation 
$\psi_{\e} := \frac{\varphi_{\e} - 1}{a(\e)}$, we write 
\begin{align}
& dY^{\e, \varphi_{\e}}(t) + \bA Y^{\e, \varphi_{\e}}(t)dt + \left[a(\e)\bB(Y^{\e, \varphi_{\e}}(t)) + \bB(Y^{\e, \varphi_{\e}}(t), \bu) + \bB(\bu, Y^{\e, \varphi_{\e}}(t))\right]dt \notag\\
&= \frac{\e}{a(\e)} \int_{\mathbb Z} g(t, \bu_{\varphi_{\e}}^{\e}(t), z) \tilde{N}^{\e^{-1}\varphi_{\e}}(dz, dt) \notag\\
&\,\, + \int_{\mathbb Z} \{g(t, \bu_{\varphi_{\e}}^{\e}(t), z) - g(t, \bu(t), z)\} \psi_{\e}(t, z)\nu(dz)dt + \int_{\mathbb Z} g(t, \bu(t), z)\psi_{\e}(t, z) \nu(dz)dt \label{yephiepeqn}
\end{align}
 Later on, in order to identify the limit of $Y^{\e, \varphi_{\e}}$, we introduce the equation   
\begin{equation}
dY^{0}(t) + \bA Y^{0}(t)dt + \left[ \bB(Y^{0}(t), \bu) + \bB(\bu, Y^{0}(t))\right]dt =   \int_{\mathbb Z} g(t, \bu(t), z)\psi(t, z) \nu(dz)dt \label{yzeroeqn}
\end{equation}
where $\psi_{\epsilon}$ converges to $\psi$ in the sense of Condition \ref{cond2} (b), i.e., $\psi_{\epsilon}1_{\{|\psi_{\epsilon}| \le \beta/a(\e)\}} \to \psi$ in $B_2(M\pr)$ for some constant $M\pr > 0$. 
We will denote $Y^{0}$ by $\mathcal{H}^0(\psi, \nu_T)$.
The notation $\mathcal{H}^0(\varphi, \nu_T)$ will also 
be used to denote the solution of \eqref{yzeroeqn} with 
$\varphi$ in the place of $\psi$ where $\varphi \in B_2(M)$. We recall that the $M$-ball $B_2(M)$ in the space $L^2(\nu_T)$  is equipped with the weak topology of $L^2(\nu_T)$. 
\begin{proposition}\label{MDPTh1}
Given any $M > 0$, let $\varphi_{\e}, \varphi \in B_2(M)$, and $\varphi_{\e} \to \varphi$ as $\e \to 0$. Then, 
$\mathcal{H}^0(\varphi_{\e}, \nu_T) \to \mathcal{H}^0(\varphi, \nu_T)$ in $L^2([0, T]; V)\cap D([0, T]; V\pr)$. 
\end{proposition}
\begin{proof}
Consider 
\begin{align}
 \sum_{j=1}^{\infty} (\int_0^T \int_{\mathbb Z} (g(t, \bu(t), z), e_j)\varphi_{\e}(t, z) \nu(dz)dt)^2
& \le \sum_{j=1}^{\infty} \norm{(g(\cdot, \bu(\cdot), \cdot), e_j)}^2_{L^2(\nu_T)} \norm{\varphi_{\e}}^2_{L^2(\nu_T)}
\notag\\
& \le M^2 \norm{\,|g(\cdot, \bu(\cdot), \cdot)|_{_H}\,}^2_{L^2(\nu_T)}\notag\\
& \le M^2 \,T \,M_2^2 \sup_{0 \le t \le T} (1 + |\bu(t)|)^2 < \infty 
\end{align}
 Hence,  $\int_0^T \int_{\mathbb Z} g(\bu(t), z)\varphi_{\e}(t, z) \nu(dz)dt$ lies in $H$. Likewise, $\int_0^T \int_{\mathbb Z} g(\bu(t), z)\varphi (t, z) \nu(dz)dt$ lies in $H$ as well.

 Let us call $\mathcal{H}^0(\varphi^{\e}, \nu_T)$ as $Y_{\e}^0$, and  
 $\mathcal{H}^0(\varphi, \nu_T)$ as $Y^0$. Define $Z_{\e} : = Y_{\e}^0 - Y^0$.
 Then, 
 \begin{align}
 &|Z_{\e}(t)|^2 + 2 \int_0^t \norm{Z_{\e}(s)}^2 ds + 2 \int_0^t \langle \bB(Z_{\e}(s), \bu(s)) + \bB(\bu(s), Z_{\e}(s)), Z_{\e}(s)\rangle ds\notag \\
 &= 2\int_0^t \int_{\mathbb Z} (g(\bu(s), z), Z_{\e}(s)) \{\varphi_{\e}(s, z) - \varphi(s, z)\}\nu(dz)ds \,\,\,\text{so}\,\,\text{that}\notag\\
 &|Z_{\e}(t)|^2 + 2 \int_0^t \norm{Z_{\e}(s)}^2 ds \notag\\
 &\!\!\!\!\le 2 C\int_0^t |Z_{\e}(s)|\,\norm{Z_{\e}(s)}\,\norm{\bu(s)}ds  + 2 |\int_0^t \int_{\mathbb Z} (g(\bu(s), z), Z_{\e}(s)) 
 \{\varphi_{\e}(s, z) - \varphi(s, z)\}\nu(dz)ds|\notag\\
 &\!\!\!\! \le  \int_0^t\!\!\norm{Z_{\e}(s)}^2 ds + C^2\!\!\!\int_0^t |Z_{\e}(s)|^2\norm{\bu(s)}^2ds 
 + 2\,|\!\!\int_0^t \int_{\mathbb Z} (g(\bu(s), z), Z_{\e}(s)) \{\varphi_{\e}(s, z) - \varphi(s, z)\}\nu(dz)ds|.
\label{zepseqn}
 \end{align}
Consider the last term on the right side of \eqref{zepseqn}:
\begin{align*}
 &2\,|\!\!\int_0^t \int_{\mathbb Z} (g(\bu(s), z), Z_{\e}(s)) \{\varphi_{\e}(s, z) - \varphi(s, z)\}\nu(dz)ds|\\
&\le 2\,|\!\!\sup_{0 \le t \le T} (1 + |\bu(s)|)\int_0^t |Z_{\e}(s)|\int_{\mathbb Z} 
M_g(z)[|\varphi_{\e}(s, z)| + |\varphi(s, z)|]\nu(dz)ds\\
& \le  2\,|\!\!\sup_{0 \le t \le T} (1 + |\bu(s)|)M_2^{1/2}\{\int_{\mathbb Z} [|\varphi_{\e}(s, z)| + |\varphi(s, z)|]^2\nu(dz)\}^{1/2}ds\\
& \le 2\,|\!\!\sup_{0 \le t \le T} (1 + |\bu(s)|)M_2^{1/2}2T\{\int_0^t |Z_{\e}(s)|^2 ds \int_0^t \int_{\mathbb Z} [|\varphi_{\e}(s, z)|^2 + |\varphi(s, z)|^2]
\nu(dz)ds\}^{1/2}\\
&\le 2\sup_{0 \le t \le T} (1 + |\bu(s)|)M_2^{1/2}T\{ \int_0^t |Z_{\e}(s)|^2 ds + 2M^2\}
\end{align*}
Using this bound in \eqref{zepseqn}, one obtains
 \begin{align*}
&|Z_{\e}(t)|^2 \\
 &\le 4M_2^{1/2}TM^2\sup_{0 \le t \le T} (1 + |\bu(s)|) + C^2\int_0^t |Z_{\e}(s)|^2\norm{\bu(s)}^2 ds + 2\,M_2^{1/2}\sup_{0 \le s \le T}(1 + |\bu(s)|)\int_0^t |Z_{\e}(s)|^2 ds.
\end{align*}
 By the Gronwall inequality, one obtains a uniform in $\e$ bound for $\{|Z_{\e}(t)|\}$ for each fixed $t$. 
 Arguing along lines similar to that of Theorem 4.4,  one obtains relatively compactness $D([0, T]; V\pr)$ and boundedness in $L^2([0, T]; V)$ so that by the result of Lions given in Appendix A, $\{Z_{\e}\}$ converges along a sequence to a limit in the strong topology of $L^2(0, T; H)$. Let $Z$ be any limit point of $\{Z_{\e}\}$. 
 There exists a sequence $\e_n \to 0$ such that  $Z_{\e_n} \to Z$ in $L_w^2([0, T]; V)\cap D([0, T]; V\pr)$ as $n \to \infty$.
 Therefore, 
 \[
 ( Z_{\e_n}(\cdot), g(\bu(\cdot), z))  \to (Z(\cdot), g(\bu(\cdot), z)) \,\,\text{in}\,\, D([0, T]; \mathbb{R}),\text{and} 
 \]
\[
 \int_0^T( Z_{\e_n}(s), g(\bu(s), z))ds \to \int_0^T (Z(s), g(\bu(s), z))ds.
\]
 With $C$ denoting an upper bound uniformly in $\e$,  for $\sup_{0 \le s \le T} |Z_{\e}(s)|^2$, we have
 \[
 \int_0^T \int_{\mathbb Z} (g(\bu(s), z), Z_{\e}(s))^2 \nu(dz)ds \le C \sup_{0 \le s \le T} (1 + |\bu(s)|^2) M_2^2 T.
 \]
 Thus $(g(\bu(\cdot), z), Z_{\e}(\cdot)) \in L^2(\nu_T)$. Likewise, 
 $(g(\bu(\cdot), z), Z(\cdot)) \in L^2(\nu_T)$. In addition,
 \begin{equation}
 \int_0^T \int_{\mathbb Z} [(g(\bu(s), z), Z_{\e_n}(s)) - (g(\bu(s), z), Z(s))]^2 \nu(dz)ds \to 0 \label{sconv}
 \end{equation}
since 
\begin{align*}
 \int_0^T \int_{\mathbb Z} [(g(\bu(s), z), Z_{\e_n}(s)- Z(s))]^2 \nu(dz)ds
& \le \int_0^T \int_{\mathbb Z} M_z^2(1 + |\bu(s)|)^2\,|Z_{\e_n}(s)- Z(s)|^2 \nu(dz)ds\\
& \le \sup_{0 \le s \le T} (1 + |\bu(s)|)^2\,M_2^2 \int_0^T |Z_{\e_n}(s)- Z(s)|^2ds
\end{align*} 
goes to zero.
 Therefore as $\e_n \to 0$, 
 \[
 \int_0^t \int_{\mathbb Z} (g(\bu(s), z), Z_{\e_n}(s)) \{\varphi_{\e_n}(s, z) - \varphi(s, z)\}\nu(dz)ds \to 0
 \]
 by weak convergence of $\varphi_{\e}$ to $\varphi$ in $L^2(\nu_T)$. 
 Using this in \eqref{zepseqn}, we obtain that for any $\d > 0$, there exists an $m$ such that 
 if $n \ge m$, then
 \begin{equation}
 |Z_{\e_n}(t)|^2 +  \int_0^t \norm{Z_{\e_n}(s)}^2 ds \le C^2\int_0^t |Z_{\e_n}(s)|^2\norm{\bu(s)}^2ds + \d.\label{zepsineq}
 \end{equation}
Ignoring the second term on the left side, and using Gronwall inequality, we obtain
\[
|Z_{\e_n}(t)|^2 \le \d \exp\{C^2\int_0^t \norm{\bu(s)}^2\}.
\]
The arbitrariness of $\d$ yields
$\displaystyle{\sup_{0 \le t \le T} |Z_{\e_n}(t)|^2} \to 0$ as $n \to \infty$. From
this and Equation \eqref{zepsineq} we conclude that as $n \to \infty$, 
\begin{equation}
\sup_{0 \le t \le T}|Z_{\e_n}(t)|^2 + \int_0^T \norm{Z_{\e_n}(s)}^2 ds \to 0.\label{firstmdp}
\end{equation}
Given any sequence $\{\e_n\}$ which converges to $0$, there exists a subsequence of $\{\e_n\}$ 
for which \eqref{firstmdp} holds. Hence, as $n \to \infty$, we have
\[
\sup_{0 \le t \le T}|Z_{\e_n}(t)|^2 +  \int_0^T \norm{Z_{\e_n}(s)}^2 ds \to 0.
\]
\end{proof}
\begin{proposition}
Given any $M > 0$, let $\{\varphi_{\e}\}$ be a family of functions in $\mathcal{U}_{+, \e}^M$ such that for some $\beta \in (0, 1]$, 
\[
\psi_{\e} 1_{\{|\psi_{\e}| \le \frac{\beta}{a(\e)}\}} \to \psi
\]
in law as $B_2(R)$-valued random elements for some fixed $R > 0$, where 
\[
\psi_{\e} = \frac{\varphi_{\e} - 1}{a(\e)}.
\]
Define $\mathcal{H}^{\e}(\varphi_{\e}, \nu_T) := Y^{\e, \varphi_{\e}}$. Then, 
the $\mathbb{D}([0, T]; V^{\prime})\cap L^2([0, T]; V)$-valued random processes 
$\mathcal{H}^{\e}(\varphi_{\e}, \nu_T)$ converge in distribution to  $\mathcal{H}^0(\psi, \nu_T)$ as $\e \to 0$.
\end{proposition}
\begin{proof} In what follows,  we will denote $Y^{\e, \varphi_{\e}}$ by $Y^{\e}$ for simplicity of notation. Then 
\begin{align}
& dY^{\e}(t) + \bA Y^{\e}(t)dt + \left[a(\e)\bB(Y^{\e}(t)) + \bB(Y^{\e}(t), \bu(t)) + \bB(\bu(t), Y^{\e}(t))\right]dt \notag\\
&= \frac{\e}{a(\e)} \int_{\mathbb Z} g(\bu^{\e, \varphi_{\e}}(t), z) \tilde{N}^{\e^{-1}\varphi_{\e}}(dz, dt) +  \int_{\mathbb Z} g(\bu^{\e, \varphi_{\e}}(t), z)  \psi_{\e}(t, z) \nu(dz)dt
\end{align}
with $Y^{\e}(0) = 0$. 
Recall that $Y^0$ denotes the solution of 
\begin{equation}
 dY^{0}(t) + \nu\bA Y^{0}(t)dt + \left[ \bB(Y^{0}(t), \bu(t)) + \bB(\bu(t), Y^{0}(t))\right]dt =   
\int_{\mathbb Z} g(\bu(t), z)\psi(t, z) \nu(dz)dt 
\end{equation}
with $Y^0(0) = 0$. 

 {\sc Step 1:}  For any given $M > 0$, there exists an $\e_0 > 0$ such that 
  \begin{equation}
  \sup_{0 < \e \le \e_0} \left( E \sup_{t \in [0, T]} |\bu^{\e, \varphi_{\e}}(s)|^2 + 
  E \int_0^T \norm{\bu^{\e, \varphi_{\e}}(s)}^2ds\right) < C(\e_0, M) < \infty. \label{step1}
  \end{equation}
  where $C(\e_0, M)$ is independent of the control process  $\varphi_{\e} \in \mathcal{U}_{+, \e}^M$ 
 with $0 < \e \le \e_0$. 
 
 \par
 Indeed, by the energy equality, 
 \begin{align}
 &|\bu_{\varphi_{\e}}^{\e}(t)|^2 + 2 \int_0^t \norm{\bu_{\varphi_{\e}}^{\e}(s)}^2 ds \notag\\
 & = |\bu(0)|^2 + 2\int_0^t \langle f(s), \bu_{\varphi_{\e}}^{\e}(s)\rangle ds + 2\int_0^t \int_{\mathbb Z} (g(\bu_{\varphi_{\e}}^{\e}(s), z), \bu_{\varphi_{\e}}^{\e}(s)) ({\varphi_{\e}}(s, z) - 1)\nu(dz)ds\notag\\
 & \;\;\;\;+ 2\e\int_0^t \int_{\mathbb Z} (g(\bu_{\varphi_{\e}}^{\e}(s), z), \bu_{\varphi_{\e}}^{\e}(s))\tilde{N}^{\e^{-1}\varphi_{\e}}(ds, dz) + \e^2 \int_0^t \int_{\mathbb Z} |g(\bu_{\varphi_{\e}}^{\e}(s), z)|^2 N^{\e^{-1}\varphi_{\e}}(ds, dz)\notag\\
 & =: |\bu(0)|^2 + 2(A_{\e}(t) + B_{\e}(t) + M_{\e}(t)) + D_{\e}(t).\label{uepsterms}
 \end{align}
 By Young's inequality,
 \begin{equation}
 |A_{\e}(t)| \le \frac{1}{4}\int_0^t \norm{f(s)}^2_{V\pr}ds +  \int_0^t \norm{\bu_{\varphi_{\e}}^{\e}(s)}^2ds. \label{uept1}
 \end{equation}
 Next, for any fixed $\beta > 0$,
 \begin{align*}
 B_{\e}(t) &\le a(\e) \int_0^t \int_{\mathbb Z} M_g(z) (1 + |\bu_{\varphi_{\e}}^{\e}(s)|) |\bu_{\varphi_{\e}}^{\e}(s)| |\psi_{\e}(z, s)| \nu(dz)ds \\
 & = a(\e) \int_0^t \int_{\mathbb Z} M_g(z) (1 + |\bu_{\varphi_{\e}}^{\e}(s)|) |\bu_{\varphi_{\e}}^{\e}(s)| |\psi_{\e}(z, s)| 1_{\{|\psi_{\e}(z, s)| \ge \beta/a(\e)\}} \nu(dz)ds \\
 &\;\;\;\;+ a(\e) \int_0^t \int_{\mathbb Z} M_g(z) (1 + |\bu_{\varphi_{\e}}^{\e}(s)|) |\bu_{\varphi_{\e}}^{\e}(s)| |\psi_{\e}(z, s)|1_{\{|\psi_{\e}(z, s)| < \beta/a(\e)\}}  \nu(dz)ds\\
 & =: B_{\e}^{(1)}(t) + B_{\e}^{(2)}(t).
 \end{align*}
 In order to bound  $B_{\e}^{(1)}(t)$, first consider 
 \begin{align}
 &a(\e) \int_0^t \int_{\mathbb Z} M_g(z)|\bu_{\varphi_{\e}}^{\e}(s)| |\psi_{\e}(z, s)| 1_{\{|\psi_{\e}(z, s)| \ge \beta/a(\e)\}} \nu(dz)ds \notag\\
 &\text{which by Young's inequality,  Lemma \ref{lemma1} and Lemma \ref{lemma4} becomes}\notag\\
 &\le \frac{1}{2}\int_0^t \int_{\mathbb Z} M_g^2(z)|\varphi_{\e}(z, s) - 1|\nu(dz)ds + \frac{a(\e)}{2} \int_0^t \int_{\mathbb Z} |\bu_{\varphi_{\e}}^{\e}(s)|^2 |\psi_{\e}(z, s)| 1_{\{|\psi_{\e}(z, s)| \ge \beta/a(\e)\}}\nu(dz)ds \notag\\
 &\le \frac{1}{2} \{\kappa_{_{M_g}}a^2(\e) + \frac{M}{\d}t + M_2^2t + a(\e) \sup_{0 \le s \le t}|\bu_{\varphi_{\e}}^{\e}(s)|^2\int_0^t \int_{\mathbb Z} |\psi_{\e}(z, s)| 1_{\{|\psi_{\e}(z, s)| \ge \beta/a(\e)\}}\nu(dz)ds\}\notag\\
 & \le \gamma_{_T} + \frac{M}{2} a^2(\e) \kappa_1(\beta) \sup_{0 \le s \le t}|\bu_{\varphi_{\e}}^{\e}(s)|^2\notag
 \end{align}
 for all $t \le T$ with $\gamma_T$  being a suitable constant independent of $\e$. Hence,
 \begin{align}
 B_{\e}^{(1)}(t) &\le \gamma_{_T} + \sup_{0 \le s \le t}|\bu_{\varphi_{\e}}^{\e}(s)|^2  \big(\frac{M}{2} a^2(\e) \kappa_1(\beta)\notag\\ & \,\,+ 
 a(\e)\int_0^t  \int_{\mathbb Z} M_g(z) |\psi_{\e}(z, s)| 1_{\{|\psi_{\e}(z, s)| \ge \beta/a(\e)\}}\nu(dz)ds\big)\notag\\
 & \le \gamma_T + a(\e) \sup_{0 \le s \le t}|\bu_{\varphi_{\e}}^{\e}(s)|^2  \left(\frac{M}{2} a(\e) \kappa_1(\beta) + 
 \Gamma_{_{M_g}}( \beta)  + (1 + t) \theta(\e)\right)\label{uept2.1}
 \end{align}
 by using Lemma \ref{lemma2} with $\zeta = 1$.
 In order to bound $B_{\e}^{(2)}(t)$, first 
 note that 
 \begin{align}
& \int_0^t \int_{\mathbb Z} M_g(z) |\bu_{\varphi_{\e}}^{\e}(s)| |\psi_{\e}(z, s)|1_{\{|\psi_{\e}(z, s)| < \beta/a(\e)\}}  \nu(dz)ds
\notag\\
&\le \frac{1}{2}\int_0^t \int_{\mathbb Z}\left(M_g^2(z) + |\bu_{\varphi_{\e}}^{\e}(s)|^2 |\psi_{\e}(z, s)|^2\right) 1_{\{|\psi_{\e}(z, s)| < \beta/a(\e)\}} \nu(dz)ds\notag\\
& \le \frac{1}{2} \left(M_2^2 T +  \sup_{0 \le s \le t}|\bu_{\varphi_{\e}}^{\e}(s)|^2 M \kappa_2(\beta)\right),\, \text{by Lemma \ref{lemma1} (iii), and hence} \notag 
 \end{align}
 \begin{equation}
 B_{\e}^{(2)}(t) \le \frac{a(\e)}{2}\left(M_2^2 T + \sup_{0 \le s \le t}|\bu_{\varphi_{\e}}^{\e}(s)|^2 (M \kappa_2(\beta) + \Gamma_{_{M_g}}(\beta) + (1+t)\theta(\e) +
 \rho_{_{M_g}}(\beta)\sqrt{t})\right) \label{uept2.2}
 \end{equation}
 by Lemma \ref{lemma2} (ii). Combining \eqref{uept2.1} and \eqref{uept2.2}, 
 \begin{equation}
B_{\e}(t)  \le  \gamma_{_T} + \frac{a(\e)}{2}M_2^2 T + a(\e)\eta(\beta, \epsilon, T)(\sup_{0 \le s \le t}|\bu_{\varphi_{\e}}^{\e}(s)|^2) \label{uept3}
 \end{equation}
 where $\eta(\beta, \epsilon, T) = \frac{M}{2} a(\e) \kappa_1(\beta) +  \Gamma_{M_g}( \beta)  + (1 + T) \theta(\e) +M \kappa_2(\beta) + \rho_{M_g}(\beta)\sqrt{T}$. 
 
 \sni
Define the stopping time $\tau_N^{\e} = \inf \{t \in [0, T]: |\bu_{\varphi_{\e}}^{\e}(t)|^2 \ge N\}$. Then, 
$\{M_\e(t \wedge \tau_N^{\e}): 0 \le t \le T\}$ is a martingale and 
\begin{align}
&E\sup_{0 \le s \le t} |M_{\e}(s \wedge \tau_N^{\e})| \notag\\
&\le \e E \left[\left(\int_0^{t\wedge \tau_N^{\e}}\int_{\mathbb Z} M_g^2(z)(1 + |\bu_{\varphi_{\e}}^{\e}(s)|)^2|\bu_{\varphi_{\e}}^{\e}(s)|^2\varphi_{\e}(z, s)\nu(dz)ds\right)^{1/2}\right] \notag\\
& \le \e E\left[ \sup_{0 \le s \le t}(1 + |\bu_{\varphi_{\e}}^{\e}(s\wedge \tau_N^{\e})|^2) 
\left(\int_0^{t}\int_{\mathbb Z} M_g^2(z) \varphi_{\e}(z, s)\nu(dz)ds\right)^{1/2}\right] \notag\\
&\le  \e \{\kappa_{_{M_g}}(\d, M, t)a^2(\e) + \frac{M}{\d}t\}^{1/2} E\left[ \sup_{0 \le s \le t}(1 + |\bu_{\varphi_{\e}}^{\e}(s\wedge \tau_N^{\e})|^2)\right] \label{uept4}
\end{align}
by Lemma \ref{lemma4}.  As before, 
\begin{align}
E \sup_{0 \le s \le t} |D_{\e}(s \wedge \tau_N^{\e})| & \le 2\e E \int_0^t \int_{\mathbb Z} M_g^2(z)(1 + |\bu_{\varphi_{\e}}^{\e}(s \wedge \tau_N^{\e})|^2)\varphi_{\e}(z, s)\nu(dz)ds\notag\\
& \le 2\e\{\kappa_{_{M_g}}(\d, M, t)a^2(\e) + \frac{M}{\d}t\}(1 +  E (\sup_{0 \le s \le t} |\bu_{\varphi_{\e}}^{\e}(s \wedge \tau_N^{\e})|^2)).\label{uept5}
\end{align}
Taking supremum over $[0, T\wedge \tau_N^{\e}]$ term-by-term and then taking expectation on the left side of  \eqref{uepsterms}, and using
 the estimates \eqref{uept1} to \eqref{uept5}, we conclude that 
 \begin{align}
 &\frac{1}{2}E \sup_{0 \le t \le T} |\bu_{\varphi_{\e}}^{\e}(t \wedge \tau_N^{\e})|^2 +  E \int_0^{T \wedge \tau_N^{\e}}\norm{\bu_{\varphi_{\e}}^{\e}(s)}^2ds \notag\\
 &\le K_T \label{uept6}
 \end{align}
 where $K_T$ is a constant which is uniform for all $N \in \mathbb{N}$ and for all $\e$ small enough. Hence 
 $\tau_N^{\e} \to \infty$ a.s.
 Letting $N \to \infty$ in \eqref{uept6}, the bound \eqref{step1} is obtained.

 \ni
 {\sc Step 2:} Using Step 1, we claim that there exists an $\e_1 > 0$ such that 
  \begin{equation}
  \sup_{0 < \e \le \e_0} \left( E \sup_{t \in [0, T]} |Y^{\e, \varphi_{\e}}|^2 + 
  E \int_0^T \norm{Y^{\e, \varphi_{\e}}(s)}^2ds,\right) < C(\e_1) < \infty.\label{step2}
  \end{equation}
  where $C(\e_1, M)$ is independent of the control process  $\varphi_{\e} \in \mathcal{U}_{+, \e}^M$ 
 with $0 < \e \le \e_1$. 
 
 \sni
 To show this, let us denote $Y^{\e, \varphi_{\e}}$ as $Y_{\e}$ here. The energy equality for $Y^{\e}$ is given by
 \begin{align}
 &|Y^{\e}(t)|^2 + 2  \int_0^t \norm{Y^{\e}(s)}^2ds + 2\int_0^t \langle B(Y^{\e}(s), \bu(s)), Y^{\e}(s)\rangle ds\notag\\
 & =  \frac{2\e}{a(\e)} \int_0^t \int_{\mathbb Z} (g(\bu_{\varphi_{\e}}^{\e}(s), z), Y^{\e}(s))\tilde{N}^{\e^{-1}\varphi_{\e}}(dz, ds) 
 + \frac{{\e}^2}{a^2(\e)} \int_0^t \int_{\mathbb Z} |g(\bu_{\varphi_{\e}}^{\e}(s), z)|^2 dN^{{\e}^{-1}\varphi_{\e}}\notag\\
&\,\, + 2\int_0^t \int_{\mathbb Z} (g(\bu_{\varphi_{\e}}^{\e}(s), z) - g(\bu(s), z), Y_s^{\e}) \psi_{\e}(s, z)\nu(dz)dt + 2\int_0^t\int_{\mathbb Z} (g(\bu(s), z), Y_s^{\e})\psi_{\e}(s, z) \nu(dz)ds\notag\\
& : = M_{\e}(t) + A_{\e}(t) + B_{\e}(t) + C_{\e}(t).  \label{enerYep}
\end{align} 
 We will bound each of the terms in \eqref{enerYep}. First, consider the absolute value of the third  term on the left side, namely, 
 \begin{align}
 | \int_0^t \langle B(Y^{\e}(s), \bu(s)), Y^{\e}(s)\rangle ds|  & \le \int_0^t | Y^{\e}(s)| \,\norm{\bu(s)} \norm{Y^{\e}(s)}ds\notag\\
 & \le  \int_0^t \norm{Y^{\e}(s)}^2 ds + 2 \int_0^t |Y^{\e}(s)|^2 \norm{\bu(s)}^2 ds\notag 
 \end{align} 
 so that 
 \begin{equation}
 \sup_{0 \le r \le t} | 2\int_0^r \langle B(Y^{\e}(s), \bu(s)), Y^{\e}(s)\rangle ds| \le   \int_0^t \norm{Y^{\e}(s)}^2 ds + 8 \int_0^t  \sup_{0 \le r \le s} |Y^{\e}(r)|^2  |\bu(s)|^2 ds \label{Yep1}
 \end{equation}
 by Young's inequality. Next, 
 \begin{align}
 E \sup_{0 \le s \le t}|M_{\e}(s)| & \le \frac{4\sqrt{\e}}{a(\e)} E\left[ \left(\int_0^t \int_{\mathbb Z} M_g(z)^2 (1 + |\bu_{\varphi_{\e}}^{\e}(s)|^2) \,|Y^{\e}(s)|^2 \varphi_{\e}(z, s)\nu(dz)ds\right)^{1/2}\right]\notag\\
 & \le  \frac{2\sqrt{\e}}{a(\e)}E \sup_{0 \le s \le t}|Y^{\e}(s)|\sup_{0 \le s \le t}(1 + |\bu_{\varphi_{\e}}^{\e}(s)|) 
 \left(\int_0^t \int_{\mathbb Z}M_g(z)^2 \varphi_{\e}(z, s)\nu(dz)ds\right)^{1/2}\notag\\
 & \le \frac{\sqrt{\e}}{a(\e)}  \left(E \sup_{0 \le s \le t}|Y^{\e}(s)|^2 + \kappa_{M_g}(a^2(\e) + t) E \sup_{0 \le s \le t}(1 + |\bu_{\varphi_{\e}}^{\e}(s)|)^2\right)
 \label{Yep2}
 \end{align}
 by Lemma \ref{lemma4}.
 Likewise, we obtain the upper bounds
 \begin{align}
 |A_{\e}(t) | &\le \frac{2\e^2}{a^2(\e)}\int_0^t \int_{\mathbb Z} M_g(z)^2 (1 + |\bu_{\varphi_{\e}}^{\e}(s)|^2)\e^{-1} \varphi_{\e}(z, s) \nu(dz)ds\notag\\
 & \le \frac{\e}{a^2(\e)} (1 + \sup_{0 \le s \le t} |\bu_{\varphi_{\e}}^{\e}(s)|^2)\int_0^t \int_{\mathbb Z} M_g(z)^2 \varphi_{\e}(z, s)\nu(dz)ds\notag\\
 & \le \frac{\e}{a^2(\e)} (1 + \sup_{0 \le s \le t} |\bu_{\varphi_{\e}}^{\e}(s)|^2)\kappa_{M_g}(a^2(\e) + t) \label{Yep3}
 \end{align}
 and by Lemma \ref{lemma2}, 
 \begin{align}
 |B_{\e}(t)| & \le a(\e) \int_0^t \int_{\mathbb Z} L_g(z) |Y^{\e}(s)|^2 |\psi_{\e}(z, s)| \nu(dz)ds\notag\\
 & =  a(\e)  \int_0^t |Y^{\e}(s)|^2\int_{\mathbb Z} L_g(z)  |\psi_{\e}(z, s)| \nu(dz)ds\notag\\
 & \le  a(\e)  (\rho_{_{L_g}}(\beta) \sqrt{T} + (1 + T) \theta(\e)) \sup_{0 \le s \le t} |Y^{\e}(s)|^2.\label{Yep4}
 \end{align}
 Finally, by Young's inequality and Lemma \ref{lemma2}, we have for any fixed $\delta > 0$,
 \begin{align}
 |C_{\e}(t)| & \le \int_0^t \int_{\mathbb Z}  M_g(z) \{\frac{2}{\delta}(1 + |\bu(s)|)^2) + \delta |Y^{\e}(s)|^2\} |\psi_{\e}(z, s)|\nu(dz)ds\notag\\
 & \le (\frac{2}{\delta}\sup_{0 \le s \le t} (1 + |\bu(s)|)^2) + \delta \sup_{0 \le s \le t} |Y^{\e}(s)|^2)\int_0^t   \int_{\mathbb Z} M_g(z)|\psi_{\e}(z, s)|\nu(dz)ds\notag\\
 & \le  [\frac{2}{\delta}\sup_{0 \le s \le t} (1 + |\bu(s)|)^2) + \delta \sup_{0 \le s \le t} |Y^{\e}(s)|^2] (\rho_{_{M_g}}(\beta) \sqrt{T} + (1 + T) \theta(\e)). \label{Yep5}
 \end{align}
 Define the stopping time the stopping time $\sigma_N^{\e} = \inf\{t: |Y^{\e}(t)| \ge N\}$. 
 Take supremum term-by-term over the interval $[0, t \wedge \sigma_N^{\e}]$, and then expectation in \eqref{enerYep}.  Using \eqref{Yep1} to \eqref{Yep5} in \eqref{enerYep} upto the stopping time 
 $ \sigma_N = \inf\{t: |Y^{\e}(t)| \ge N\}$, and proceeding as in Step 1, we obtain \eqref{step2} by  Gronwall 
 Lemma.
 
 {\sc Step 3:} Let $\{\varphi_{\e}\}$ be as in the statement of the Theorem.
 By Step 2, we conclude that $\{Y^{\e}: 0 < \e < \e_1\}$ is tight as random elements taking values in  $L^2(0, T; V) \cap D([0, T]; V\pr))$ by using arguments similar to those of Theorem 4.3 and 4.4. The tightness argument in $D([0, T]; V\pr)$ is briefly outlined below.

We first need to prove that for each fixed $t $ in a dense subset $J$ of $[0, T]$, $\{Y^{\e}_t: 0 < \e < \e_1\}$ is  tight as $V\pr$-valued random variables.  Consider any sequence $\{Y^{\e_n}_t\}$ from the family. Since $V\pr$ is a separable Hilbert space,
 it suffices to show the following:
 
 (i) For each $t \in J$, and every $h \in V\pr$, the laws of the real random variables $\{(Y^{\e_n}_t, h)_{V\pr}\}$ forms a tight 
 family, and 
 
 (ii) For each $\d > 0, \rho > 0$, and $t \in J$, there exists a finite dimensional subspace $F$ of $V\pr$ such that 
 $\limsup_{\n \to \infty} P\{d(Y^{\e_n}_t, F) > \rho\} < \d$ where $d$ denotes distance in $V\pr$. 
 
 Since $\sup_{n}\mathbb{E}[|Y^{\e_n}_t|^2]$ is finite, 
 condition (i) follows. To verify condition (ii), set
 $F_k = [{\rm{Span}} \,\{e_1, \cdots, e_k\}]$ where $\{e_j\}$ are the eigenfunctions of the Stokes operator.  
 
 From the estimate 
 $\sup_{n}\mathbb{E}[\sup_{0 \le t \le T}|Y^{\e_n}_t|^2]< \infty$, we have
 \begin{align*}
 &\lim_{k \to \infty} \sup_{n} P\{d(Y^{\e_n}_t, F_k)  > \rho\} \\
 &= \lim_{k \to \infty} \sup_{n} 
 P\{\sup_{0 \le t \le T}\norm{(I - P_k) Y^{\e_n}_t}_{V\pr} > \rho\} \\
 & = 0, \,\,\,\,\text{where $P_k$ denotes projection in $V\pr$ to $F_k$}.
 \end{align*}
 Next, let $\{\tau_n\}$ be a family of stopping times each taking values in $[0, T]$ such that there exists a $\d_n > 0$ with $\tau_n +\d_n \le T$
 for each $n$ and $\d_n \to 0$. We need to show the Aldous condition that 
 $\norm{Y^{\e_n}_{ \tau_n +\d_n } - Y^{\e_n}_{ \tau_n}}_{V\pr} \to 0$ in probability as $n \to \infty$.  An inspection of the proof of the estimate in Step 2 shows that 
 \[
 \lim_{n \to 0} \mathbb{E} \left[\sup_{0 \le s \le \d_n} |Y^{\e_n}_{ \tau_n + s } - Y^{\e_n}_{ \tau_n}|^2\right] = 0
 \]
 which implies the desired convergence in probability.

{\sc Step 4:} Let $\{\eta_{\e}(t)\}$ denote the solution of 
the equation 
\begin{equation}
d\eta_{\e}(t) + \bA \eta_{\e}(t) dt = \frac{\e}{a(\e)} \int_{\mathbb Z} g(u^{\e, \varphi_{\e}}(t), z) \tilde{N}^{\e^{-1}\varphi_{\e}}(dz, dt) \label{etaeqn}
\end{equation}
 with $\eta_{\e}(0) = 0$. As before, one establishes the existence and uniqueness of solutions to \eqref{etaeqn} for $\e > 0$ small enough 
 with values in $D([0, T]; V\pr) \cap L^2(0, T; V)$.  Besides, it is straightforward to show that
 \begin{equation}
 \lim_{\e \to 0} \mathbb{E}\{\sup_{t \in [0, T]} |\eta_\e(t)|^2 + 
 \int_0^T \norm{\eta_{\e}(s)}^2ds\} = 0. \label{etaest}
 \end{equation}
 
Define the space $\Lambda := \mathbb{D}([0, T]; V\pr)\cap L^2([0, T]; V)$.   Define the space 
$\Xi := \Lambda \times \Lambda \times \bar{\mathbb{M}} \times B_2(R)$.
Then $(Y^{\e}, \eta_{\e}, \e N^{\e^{-1}\varphi_{\e}}, \psi_{\e}1_{\{|\psi_{\e}| \le \beta/a(\e)\}})$ 
 is tight as $\Xi$-valued random elements. Let $(\tilde{Y}^0, 0,  0, \psi)$ 
 be a limit point of the above tight family. We will show that 
 $\tilde{Y}^0$ has the same law as $Y^0$. Let $\{\e_n\}$ be a sequence 
 which converges to zero, and 
 $(Y^{\e_n}, \eta_{\e_n}, \e N^{\e_n^{-1}\varphi_{\e_n}}, \psi_{\e_n}1_{\{|\psi_{\e_n}| \le \beta/a(\e_n)\}})$  be  weakly convergent to $(\tilde{Y}^0, 0, 0, \psi)$. 
 
Let $\mathcal{L}$ denote the law of a random element. By the Skorohod representation theorem, by a change of stochastic basis, one obtains the existence of a sequence of random elements 

$(\overline{Y}^{\e_n}, \overline{\eta}_{\e_n}, \e \hat{N}^{\e^{-1}\varphi_{\e}}, \overline{\psi}_{\e_n}1_{\{|\overline{\psi}_{\e_n}| \le \beta/a(\e_n)\}})$  which converges a.s.  to $(\overline{Y}^0, 0,  0, \overline{\psi})$ such that 
\[ \mathcal{L}\,
(\overline{Y}^{\e_n}, \overline{\eta}_{\e_n}, \e_n \hat{N}^{\e_n^{-1}\varphi_{\e_n}}, \overline{\psi}_{\e_n}1_{\{|\overline{\psi}_{\e_n}| \le \beta/a(\e_n)\}})
= \mathcal{L}\,
(Y^{\e_n}, \eta_{\e_n}, \e_n \hat{N}^{\e_n^{-1}\varphi_{\e_n}}, \psi_{\e_n}1_{\{|\psi_{\e_n}| \le \beta/a(\e_n)\}})
\]
for each $n$,  and 
$
\mathcal{L}\,
(\overline{Y}^0, \overline{\psi})
= \mathcal{L}\,(\tilde{Y}^0, \psi).
$
For notational simplicity, we will write 

$(\overline{Y}^{\e_n}, \overline{\eta}_{\e_n}, \e_n \hat{N}^{\e_n^{-1}\varphi_{\e_n}}, \overline{\psi}_{\e_n}1_{\{|\overline{\psi}_{\e_n}| \le \beta/a(\e_n)\}})$ as 
$(\overline{Y}_n, \overline{\eta}_n, \e_n \hat{N}_n, \overline{\psi}_n1_{\{|\overline{\psi}_n| \le  \beta/a(\e_n)\}})$ and the limit as 

$(\overline{Y}^0, 0, 0, \overline{\psi})$.

Define the process $X_n := \overline{Y}_n - \overline{\eta}_n$. Then $X_n$ satisfies 
\begin{align}
&dX_n(t) + \{\bA(X_n(t)) + a(\e)\bB(X_n(t) + \overline{\eta}_n(t)) + \bB(X^n(t) + \overline{\eta}_n(t), \bu(t)) + \bB(\bu(t), X_n(t) + \overline{\eta}_n(t))\}dt \notag\\
& = \int_{\mathbb Z} g(a(\e)(X_n(t) + \overline{\eta}_n(t)) + \bu(t), z)\overline{\psi}_n(z, t)\nu(dz)dt. \label{xneqn}
\end{align}
with $X_n(0) = 0$.
Define $\hat{Y}^0$ by 
\[
 d\hat{Y}^0(t) + \{\bA(\hat{Y}^0(t)) + \bB(\hat{Y}^0(t), \bu(t)) + \bB(\bu(t), \hat{Y}^0(t))\}dt = \int_{Z} g(t, \bu(t), z)\overline{\psi}(t, z)\nu(dz)dt
\]
We also know that 
\begin{equation}
 \sup_{0 < \e \le \e_0} \left( E \sup_{t \in [0, T]} |{X}_n(t) - \hat{Y}^0(t)|^2 + 
  E \int_0^T \norm{{X}_n(s) - \hat{Y}^0(s)}^2ds,\right) \to 0\label{endeqn}
\end{equation}
almost surely as $\e \to 0$. Using \eqref{etaest} along with \eqref{endeqn}, we proceed as in the last part of the proof of Theorem 4.5 to argue that
${X}_n$ converges in $\Xi$ to $\hat{Y}^0$. Since $\overline{Y}_n \to \overline{Y}^0$ a.s., we obtain that $\overline{Y}^0 = \hat{Y}^0$. 
Thus $\overline{Y}_n \to \hat{Y}^0$ which completes the proof.
\end{proof}

\appendix
\section{Examples of SPDEs}
The class of SPDEs of the form \eqref{epsiloneqn} is motivated by the following equations that arise 
in fluid dynamics. 

\smallskip
\noindent
{\sc 1. Two-dimensional  Navier-Stokes system}

 Let $G$ be a bounded open domain in $\mathbb{R}^{2}$ with a smooth boundary $\partial
G$. For $t\in [0,T]$, consider the Navier-Stokes system  for a viscous
incompressible fluid flow with no-slip condition at the boundary:
\begin{align}\label{nse}
\frac{\partial\mathbf{u}}{\partial t}+(\mathbf{u}\cdot\nabla)\mathbf{u} 
-\Delta\mathbf{u}+\nabla p &= 0, \qquad
\nabla\cdot\mathbf{u} =0  
\end{align}
with initial and boundary data given by
\begin{align*}
\mathbf{u}(0,x)  &  =\mathbf{u}_{0}(x)\,\,\forall\,\,x\in G\\
\mathbf{u}(t,x)  &  =0\,\,\forall \,x\in\partial G,\, \text{and}\,\, t\geq0
\end{align*}
In the above, $p$ denotes the pressure field and is real-valued. The viscosity coefficient is taken to be $1$. 

The following
characterizations of the spaces ${H}$ and ${V}$ are
well-known:%
\begin{align*}
{H}  &  =\{\mathbf{u}\in L^{2}(G);\nabla\cdot\mathbf{u}%
=0,\,\,\mathbf{u}\cdot\mathbf{n}|_{\partial G}=0\},\\
{V}  &  =\{\mathbf{u}\in W_{0}^{1,2}(G):\nabla\cdot\mathbf{u}=0\}
\end{align*}
where $\mathbf{n}$ is the outward normal on $\partial G$.
 Here, derivatives are to be understood in the weak sense, and $\bu \cdot \mathbf {n}$ is taken in the sense of trace operators.

Let ${V}^{\prime}$ be the dual of ${V}$. We will denote the
norm in ${H}$ by $|\cdot|$, and the inner product in ${H}$ by
$(\cdot,\cdot)$. We have the dense, continuous  embedding (see
\cite{Tem}):%
\[
{V}\subset_{\rightarrow}{H}={H}^{^{\prime}}%
\subset_{\rightarrow}{V}^{^{\prime}}.
\]
 Define the linear
operator $\mathbf{A}:{\mathcal{D}}(\mathbf{A})\rightarrow {H}$ by
$\mathbf{A}\mathbf{u}=-\Delta\mathbf{u}$. Since ${V}={\mathcal{D}%
}(\mathbf{A}^{1/2})$, we endow ${V}$ with the norm $\left\Vert
\mathbf{u}\right\Vert =\left\vert \mathbf{A}^{1/2}\mathbf{u}\right\vert $. The
${V}$-norm is equivalent to the $W^{1,2}$-norm by the Poincar\'{e}
inequality.  The pairing between ${V}$ and its dual
${V}^{\prime}$ is indicated by $\left\langle \cdot,\cdot\right\rangle $.
The operator $\mathbf{A}$ is known as the Stokes operator and is positive,
self-adjoint, with compact resolvent. 

Let ${\mathcal{V}}$ denote the space of all smooth, solenoidal 2-dimensional vector valued functions with compact support in $G$.
Define $b(\cdot,\cdot,\cdot): {\mathcal{V}}\times {\mathcal{V}}\times {\mathcal{V}}%
\mapsto\mathbb{R}$ by
\[
b(\mathbf{u},\mathbf{v},\mathbf{w})={\sum_{i,j=1}^{2}\int_{G}{u_{i}}}\left(
x\right)  {{{\frac{\partial v_{j}}{\partial x_{i}}}}}\left(  x\right)
{{w_{j}}}\left(  x\right)  {\,dx.}%
\]
This allows us to define $\mathbf{B}: {V}\times {V}%
\mapsto {V}^{\prime}$ as the continuous bilinear operator satisfying%
\[
\left\langle \mathbf{B}(\mathbf{u},\mathbf{v}), \mathbf{w}\right\rangle
=b(\mathbf{u},\mathbf{v},\mathbf{w})\,\,\mathrm{{for}\,\,{all}\,\,}%
\mathbf{u,v,w}\in {\mathcal{V}}\mathrm{.}%
\]
Note that $b(\mathbf{u},\mathbf{v},\mathbf{w})=-b(\mathbf{u},\mathbf{w}%
,\mathbf{v})$. By H\"older's inequality, Sobolev embedding and interpolation inequality, one obtains 
\begin{equation}
 \langle\bB (\mathbf{u}, \mathbf{v}), \mathbf{w}\rangle^2 \le C  |\bu|\,\norm{\bu}\,|\bv| \,\norm{\bv}\,\norm{w}^2 \,\,\text{for all} \,\, \bu, \bv, \bw \in V \label{estimB}
\end{equation}
where $C$ denotes a constant that depends on $G$.

We will denote $\mathbf{B}(\mathbf{u},\mathbf{u})$ by
$\mathbf{B}(\mathbf{u})$, and $\mathbf{B}(\mathbf{u})$ satisfies the
 estimate:
\begin{equation}
\left\Vert \mathbf{B}(\mathbf{u})\right\Vert _{{V}^{\prime}}%
\leq C\,|\mathbf{u}|\,||\mathbf{u}|| \label{estim}%
\end{equation}
The local Lipschitz property of $\bB$ that appears in hypotheses {\bf{H.4}} is readily obtained 
from \eqref{estimB}. Next, we proceed to establish the local monotonicity 
property in hypothesis {\bf H.7}.
\begin{proposition} \label{prop1} For a given $r > 0$, let $B_r$ denote the $L^4(G)$ ball in $V$:
$B_r = \{ \bv \in V: ||\bv||_{L^4(G)} \le r\}$. Define  the nonlinear operator $F$ on $V$ by
$F(\bu):= - \bA \bu - \bB(\bu)$.  Then for $0 < \e < \frac{\l_1}{2 L^2_2}$ where $\l_1$ is the smallest eigenvalue of $\bA$ and $L_2$ is the $L^2(\nu)$ norm of the function $L_g$, the pair $(F, \e g)$
is monotone in $B_r$: ie. for any $\bu\in V$ and $\bv \in B_r$, if $\bw$ denotes $ \bu - \bv$, then for all $t \in [0, T]$,
\begin{equation}
\langle F(\bu) - F(\bv), \bw \rangle  - r^4\,|\bw|^2 + \e\int_{\mathbb Z} |g(t, \bu, z) - g(t, \bv, z)|^2\nu(dz) \le 0.
\end{equation}
\end{proposition}
\begin{proof} First, it is clear that $\langle\bA\bw, \bw\rangle = ||\bw||^2$.
Using the bilinearity of the operator $B$, it follows
that
\[\langle\bB(\bu), \bw\rangle = -\langle\bB(\bu, \bw), \bv\rangle
\]
Likewise, $\langle\bB(\bv), \bw\rangle = -\langle\bB(\bv, \bw), \bv\rangle$.
Using  these two equations, one obtains
\[\langle\bB(\bu)-\bB(\bv), \bw\rangle = -\langle\bB(\bw), \bv\rangle.\]
From the H\"older inequality, one has
\begin{eqnarray}
|\langle\bB(\bu) - \bB(\bv), \bw\rangle| &\le &||\bw||_{L^4(G)}||\bw||\,||\bv||_{L^4(G)}\nonumber\\
&\le &|\bw|^{1/2}\,||\bw||^{3/2}\,||\bv||_{L^4(G)}\nonumber\\
&\le &\frac{1}{2}||\bw||^2 + \frac{27}{32}|\bw|^2\,||\bv||^4_{L^4(G)}\label{27}
\end{eqnarray}
where the last estimate accrues by Young's inequality. 
The definition of the operator $F$ in the theorem yields
\begin{equation}
\langle F(\bu) - F(\bv), \bw \rangle \le -\frac{1}{2}||\bw||^2 + r^4 |\bw|^2\label{mon}
\end{equation}
since $||\bv||_{L^4(G)} \le r$. The proof is finished upon using hypothesis {\bf H.6} and the choice of $\e$.
\end{proof}
From the proposition it follows that hypothesis {\bf H.7} holds with the norm $\norm{}_{_0} = \norm{}_{L^4(G)}$. Indeed, hypotheses {\bf H.1} to {\bf H.4}, and {\bf H.7} 
that involve the operators $\bA$ and $\bB$ have all been shown to hold.
Hypotheses {\bf H.5}, {\bf H.6} and {\bf H.8} are conditions imposed on the noise coefficient. 

We assume that $\mathbf{u}_{0}$ is ${H}$-valued. Let $\Pi$ denote the
 projection of $L^{2}(G)$ into ${H}$. Applying $\Pi$ 
to each term of the Navier-Stokes system, and invoking the Helmholtz-Hodge decomposition
of $L^{2}(G)$ into solenoidal and irrotational components, we  write
the system \eqref{nse} with the body force $f$ and the noise term as%
\begin{equation}
d\mathbf{u}\left(  t\right)  +[\mathbf{A}\,\mathbf{u}\left(  t\right)
+\mathbf{B}(\mathbf{u}\left(  t\right)  )]\,dt=  f(t)dt +  \int_{\mathbb Z} g(t, \bu(t-), z) \tilde{N}^{1/\e}(dt, dz)
\label{de1}%
\end{equation}

\smallskip
\noindent
\leftline{\sc 2. Navier-Stokes equation with shear thickening viscosity}

Let $G \subset \mathbb{R}^2$ be a bounded domain with a smooth boundary. Let $\bv = \bv(x,t)$
denote the velocity and $p= p(x,t)$, the pressure. Consider the following 2D Navier-Stokes equations with shear-thickening viscosity, namely
\begin{equation}\label{NSnlvE}
\frac{\partial\bv}{\partial t} - \nu_0 \Delta \bv - \nu_1 \nabla\cdot (|\nabla \bv|^{q-2}\nabla \bv) + (\bv \cdot)\bv = - \nabla p 
\end{equation}
with divergence-free condition $\nabla \cdot \bv = 0$,
and with no-slip boundary and initial condition as below:
\[
\bv(x,t) = 0 \quad \text{on}\quad \partial G; \quad \bv(x, 0) = \bv_0(x) \quad \text{on}\quad G.\]
Here $\nu_i > 0$ and $q\geq 2$. In addition to the spaces introduced above in Example 1,  define $V_{1, q} := \{\bv \in W^{1, q}_0(G): \nabla \cdot \bv= 0\}$ with $V^{\prime}_{1, q}$ as its dual. 
Projecting \eqref{NSnlvE} to $H$ as before, we have the following abstract evolution equation:
\begin{equation}\label{NSnlvE2}
\frac{\partial\bv(t)}{\partial t}
+ (\nu_0\bA{\bv} + \nu_1\bA_1(\bv) + \bB(\bv)) = 0
\end{equation}
with initial condition $\bv(0) = \bv_0$ where $\bA$ and $\bB$ are as before, and the nonlinear viscosity operator $\bA_1$ defined by
$\bA_1: V_{1, q} \to V^{\prime}_{1, q}$ by $\bA_1(\bv) = - \Pi_H\nabla \cdot (|\nabla \bv|^{q-2}\nabla \bv)$. 
When perturbed by a small jump noise, the coefficients of \eqref{NSnlvE2} satisfy the local monotonicity hypothesis and Hypotheses $\bf{H}$ for a suitable noise coefficient as in Example 1. A similar setup can be used for Navier-Stokes equations with other types of nonlinear viscosities \cite{Lad}.

\smallskip
\noindent
\leftline {\sc 3.  Shell model of turbulence:} 

Consider the system of equations given by
\begin{equation}
\frac{du_n}{dt} + \rho k_n^2 u_n + i(a k_{n+1}u_{n+1}^*u_{n+2}^* + b k_n u_{n-1}^* u_{n+1}^* + ck_{n-1}u_{n-1}^* u_{n-2}^*) = f_n 
\end{equation}
for $ n = 1, 2, \cdots $, and $u_n \in \mathbb{C}$ with $u_n^*$ being the complex conjugate of $u_n$. We have  the boundary conditions $u_{-1} = u_0 = 0$. 
Here, $\rho$ is a positive constant that denotes kinematic viscosity, and $a, b, c$ are real valued parameters such that conservation of energy holds:
$a + b + c = 0$. For the standard model, $a = -1, b = \frac{1}{2},$  and $c = \frac{1}{2}$.  The above equation pertains to the shell model due to Gledger, Ohkitani and Yamada \cite{OY}, and is known as the GOY model.  

\par
The model arises when the Navier-Stokes equation is written in the Fourier space where the interaction between different modes is preserved between the nearest modes. 
To be precise, the GOY model describes a one-dimensional cascade of energies among an infinite sequence of complex velocities $\{u_n(t)\}$  on a one-dimensional sequence of wave 
numbers, 
\[
k_n = k_0 2^n, \;\;k_0 >0,\;\;n = 1, 2, \cdots
\]
where the discrete index $n$ is referred to as the "shell index" (see \cite{Fr}).  The numbers $k_n$ may also be taken as $k_n = k_0\l^n$ for $\l >1$.

For the functional analytic setup, define   the real Hilbert space $H$ by 
\begin{equation*}
H :=\Big\{u=(u_1, u_2, \ldots) \in \mathbb{C}^{\infty}: \sum_{n=1}^{\infty} |u_n|^2 < \infty\Big\}.
\end{equation*}
For every $u, v \in H$, the scalar product $(\cdot,\cdot)$ and norm $|\cdot|$ are defined on $H$ as
\[(u, v)_H = Re\ \sum_{n=1}^{\infty} u_n v\s_n, \quad |u| = \big(\sum_{n=1}^{\infty} |u_n|^2\big)^{1/2}.\]
Let us now define the space
$V :=\Big\{u\in H: \sum_{n=1}^{\infty} k_n^2|u_n|^2 < \infty\Big\},$
which is a Hilbert space equipped with the norm
$$\|u\| = \big(\sum_{n=1}^{\infty} k_n^2|u_n|^2\big)^{1/2}.$$
The linear operator $A: D(A) \rightarrow H$ is a positive definite, self adjoint linear operator defined by
\begin{equation}\label{A}
Au=((Au)_1, (Au)_2, \ldots), \ \text{where}\ (Au)_n = k_n^2 u_n, \quad\forall \,\,u\in D(A).
\end{equation}
The domain of $A$ denoted by $D(A) \subset H$, is a Hilbert space equipped with the norm
$$\|u\|_{D(A)} = |Au| = \big(\sum_{n=1}^{\infty} k_n^4|u_n|^2\big)^{1/2}, \quad\forall\,\, u\in D(A).$$
Since the operator $A$ is positive definite, we can define the power $A^{1/2}$ ,
$$A^{1/2}u = (k_1 u_1, k_2 u_2, \ldots), \quad\forall \,\,u=(u_1, u_2, \ldots).$$
Furthermore, we define the space
$D(A^{1/2}) = \Big\{u=(u_1, u_2, \ldots): \sum_{n=1}^{\infty} k_n^2 |u_n|^2 < \infty\Big\}$
which is a Hilbert space equipped with the scalar product
$$ (u, v)_{D(A^{1/2})} = (A^{1/2}u, A^{1/2}v), \quad\forall\,\, u, v\in D(A^{1/2}),$$
and the norm $\|u\|_{D(A^{1/2})} = \big(\sum_{n=1}^{\infty} k_n^2 |u_n|^2\big)^{1/2}.$
Note that $V = D(A^{1/2})$. We consider $V^{\prime} = D(A^{-1/2})$ as the dual space of $V$. Then the following dense, continuous inclusion holds
$$V\subset_{\to} H = H^{\prime}\subset_{\to} V^{\prime}$$
with the inclusion of $V$ in $H$ being compact.
We will now introduce the sequence spaces analogue to Sobolev functional spaces.
For $1\leq p <\infty$ and $s\in\mathbb{R}$
$$\mathrm{W}^{s, p} :=\Big\{u = (u_1, u_2, \ldots): \|A^{s/2}u\|_p = \big(\sum_{n=1}^{\infty}(k_n^s|u_n|)^p\big)^{1/p} < \infty\Big\},$$
and for $p=\infty$,
$\mathrm{W}^{s, \infty} :=\Big\{u = (u_1, u_2, \ldots): \|A^{s/2}u\|_{\infty} = \sup_{1\leq n<\infty} (k_n^s|u_n|) < \infty\Big\},$
where for $u\in\mathrm{W}^{s, p}$ the norm is defined as
$\|u\|_{\mathrm{W}^{s, p}} = \|A^{s/2}u\|_p.$
Here $\|\cdot\|$ denotes the usual norm in the $l^p$ sequence space. It is clear from the above definitions that $W^{1, 2} = V = D(A^{1/2})$.

It is well-known that $
  \| u \|_{l^4}^4 \leq C |u|^2 \ \|u\|^2.$
We define the bilinear operator $B(\cdot, \cdot): V \times H\rightarrow H$ as
$$B(u, v) = (B_1(u,v), B_2(u, v), \ldots),$$ where
\begin{align*}
B_n(u, v)= ik_n\big(\frac{1}{4}u\s_{n+1} v\s_{n-1} -
\frac{1}{2}(u\s_{n+1} v\s_{n+2} + u\s_{n+2} v\s_{n+1})
 + \frac{1}{8}u\s_{n-1} v\s_{n-2}\big).
\end{align*}
If $\{e_n\}_{n=1}^{\infty}$ is the standard orthonormal
basis of $H$, i.e. all the entries of $e_n$ are zero except for the
 $n$th co\"ordinate which is equal to $1$, then
\begin{align}\label{B}
B(u, v)= i\sum_{n=1}^{\infty}k_n\big(\frac{1}{4}u\s_{n+1} v\s_{n-1}
- \frac{1}{2}(u\s_{n+1} v\s_{n+2} + u\s_{n+2} v\s_{n+1})
 + \frac{1}{8}u\s_{n-1} v\s_{n-2}\big)e_n.
\end{align}

 It follows from the following lemma that $B(u, v)$ makes
sense as an element of $H$, whenever $u\in V$ and $v\in H$ or $u\in
H$ and $v\in V$. In addition,  $B(u, v)$ makes sense as an
element of $V^{\prime}$. Here we state the following lemma which has
been proved in Constantin, Levant and Titi \cite{CLT} for the Sabra
shell model, but one can also prove  similar estimates for the
GOY model (see Barbato, Barsanti, Bessaih, and Flandoli\cite{BBBF}).
\begin{lemma}\label{lem1}
(i) There exist constants $C_1 >0, C_2 >0$,
\begin{equation}
|B(u, v)| \leq C_1 \|u\| |v|, \quad\forall u\in V, v\in H,
\end{equation}
and
\begin{equation}
|B(u, v)| \leq C_2 |u| \|v\|, \quad\forall u\in H, v\in V.
\end{equation}
(ii) $B: H\times H\rightarrow V^{\prime}$ is a bounded bilinear operator and for a constant $C_3 > 0$
\begin{equation}
\|B(u, v)\|_{V^{\prime}} \leq C_3 |u| |v|, \quad\forall u, v\in H.
\end{equation}
(iii) $B: H\times D(A)\rightarrow V$ is a bounded bilinear operator and for a constant $C_4 > 0$
\begin{equation}
\|B(u, v)\|_{V} \leq C_4 |u| |Av|, \quad\forall u\in H, v\in D(A).
\end{equation}
(iv) For every $u\in V$ and $v\in H$
\begin{equation}\label{1}
(B(u, v), v) = 0.
\end{equation}
\end{lemma}

With above functional setting, we introduce  the stochastic shell equation for the GOY model  on the  Hilbert space $H$ by the equation
\begin{align}\label{goy1}
& du + \big[\rho Au + B(u, u)\big] dt = f(t)dt +  \int_{\mathbb Z} g(t, \bu(t-), z) \tilde{N}^{1/\e}(dt, dz)\\
& u(0) = u_0 \notag,
\end{align}
where $u=(u_1, u_2,\ldots)\in H$, and the operators $A$ and $B$ are
defined through \eqref{A} and \eqref{B} respectively. Local monotonicity condition is easy to check. Indeed, by bilinearity of $B$, 
we obtain 
$
B(u) - B(v) = B(v, w) + B(w, v) + B(w, w)$, with $w = u - v$. Hence, 
\[
\langle B(u) - B(v), w \rangle  = \langle B(w, v), w \rangle\,\,\text{since}\,\,B(v, w, w) = B(w, w, w) = 0.
\]
Using the definition of $B$ and then applying the H\"older inequality, we obtain
\begin{align*}
|\langle B(u) - B(v), w \rangle |& = |\langle B(w, v), w \rangle |\\
& \le \norm{w}_{l^4} \norm{w} \,\norm{v}_{l^4}\\
&\le  |w|^{1/2}\norm{w}^{3/2}\,\norm{v}_{l^4} \,\,\text{since}\,\,\norm{w}_{l^4} \le |w|^{1/2}\norm{w}^{1/2}; \,\text{continuing},\\
&\le \frac{\norm{w}^2}{2} + \frac{27}{32} |w|^2\,\norm{v}_{l^4}^4
\end{align*}
by Young's inequality, as before. The rest of the proof of local monotonicity follows along lines similar to that of Proposition \ref{prop1}. The rest of the hypotheses in ${\bf H}$ are easily verified from  Lemma \ref{lem1} and basic considerations. 

\medskip
\noindent
\leftline{{\sc 4.  MHD System}}

Suppose that $G \subset \mathbb{R}^2$  is  a  bounded,
open, simply connected  domain with a smooth boundary $\partial G$.
The stochastic MHD system (see Sermange and Temam \cite{SerTem})  in the
non-dimensional form is given by
\begin{equation}
\frac{\partial v}{ \partial t} + (v\cdot\nabla)v -\frac {1}{R_e} \Delta v - S(B\cdot \nabla)B
 + \nabla(p +\frac {S|B|^2}{ 2}) = 0,\label {mhd1}
\end{equation}
where $v$ denotes  the velocity,  $B$,  the magnetic field,  $p$, the pressure field, and
 $R_e$, the Reynolds number. Let $ S:= \frac{M^2} {R_e R_m}$,
where $M$ is the Hartman number and $R_m$ is the magnetic Reynolds number.
The Maxwell equation is
\begin{equation}
\frac{\partial B}{ \partial t} + (v\cdot \nabla)B + \frac{1}{R_m}{\rm curl}({\rm curl}\,\, B)
 - (B\cdot \nabla)v = 0,\label{mhd2}
\end{equation}
where $\nabla\cdot v = 0$ and $\nabla\cdot B = 0$.

For a vector-valued function $u = (u_1, u_2)$ defined on $G$, 
$
 \text{curl}\;u =  \frac{\p u_2}{\p x} - \frac{\p u_1}{\p  y}.
$
If $\varphi$ is a scalar-valued function defined on $G$, 
$
  \text{curl}\;\varphi =  \bigl(\frac{\p \varphi}{\p  y},  - \frac{\p \varphi}{\p x}\bigr).
$

Equations (\ref {mhd1}), (\ref {mhd2}) are equipped
with the following boundary conditions:
$ v = 0$ on $\partial G$ ; $B\cdot {\bf n} = 0$ and $ {\rm curl}\,\, B = 0$ on $\partial G$,
where ${\bf n} $ is the unit outer normal on $\partial G$. The initial conditions are: 

$v(x,0) = v_0(x)$  and
$B(x,0) = B_0(x)$ for $ x\in G$.

The system needs to be written as an abstract evolution equation
which in the integral form reads is given by
\begin{equation}
\by(t) = \by(0) + \int_0^t [-\bA \by(s) - \bB(\by(s))]\,ds,\label {mhde1}
\end{equation}
where $\by$ denotes the transpose of $(v, B)$. As in Sermange and
Temam \cite{SerTem}, we will begin with the choice of function
spaces.

Set $H = H_1\t H_2 \quad {\rm and} \quad V = V_1\t V_2$, where
\begin{eqnarray*}
H_1 &=& \{\varphi\in L^2(G); \nabla\cdot \varphi = 0,\,\, \varphi\cdot {\bf n}|_{\p G} = 0\}\\
H_2 &=& H_1\\
V_1 &=& \{\varphi \in H_0^1(G): \n \cdot \varphi = 0\}\\
V_2 &=& \{\varphi \in H^1(G): \n \cdot \varphi = 0,\,\, \varphi\cdot {\bf n}|_{\p G} = 0\}
\end{eqnarray*}
where
\begin{eqnarray*}
H^1 &=& \{u\in L^2(G): \nabla u\in L^2(G)\}\\
H_0^1 &=&\{u\in L^2(G): \nabla u\in L^2(G)\,\, {\rm and}\,\, u|_{\p G} = 0\}
\end{eqnarray*}
$H_1$ and $H_2$ are equipped with the $L^2(G)$ norm. Define the  inner
product on $H$ by 

$ [\by_1, \by_2] = (v_1, v_2)_{H_1} + S(B_1, B_2)_{H_2}\,\,\text{where}\,\, \by_i = \{v_i, B_i\}$. 
The above inner product is equivalent to $(\by_1, \by_2)_H = (v_1,
v_2)_{H_1} + (B_1, B_2)_{H_2}$.
 The norm on $H$ is given by $||\by||_H = \sqrt {(y, y)_H}$.
The space $V_1$ is endowed with the inner product
$[| \varphi, \psi |]_{V_1} = (\nabla \varphi, \nabla \psi)_{L^2(G)}$.
The norm on $V_1$ is given by $||\varphi||_{V_1} = \sqrt {[| \varphi, \varphi |]_{V_1}}$.
Note that the $V_1$ norm given here is equivalent to the usual $H^1(G)$ norm since
$||\nabla \varphi||_{L^2(G)} \ge C||\varphi||_{L^2(G)}$ for a suitable $C$ by the Poincar\'e inequality.

\sni Let $V_2$ be  endowed  with the inner product $[| \varphi, \psi
|]_{V_2} = ({\rm curl}\,\varphi, {\rm curl}\,\psi)_{L^2(G)}$  so that the norm
on $V_2$ is given by $||\varphi||_{V_2} = \sqrt {[| \varphi, \varphi
|]_{V_2}}$. From Proposition 1.8 and Lemma 1.6 of \cite {Tem}, we
can conclude that the norm given by $\{||\varphi||^2_{L^2(G)} + ||{\rm curl}\,
\varphi||^2_{L^2(G)} + ||\varphi\cdot {\bf n}||^2_{H^{1/2}(\p G)}\}^{1/2}$
is equivalent to the $H^1(G)$ norm.Thus on $V_2$, the norm defined
above is equivalent to the $H^1(G)$ norm. We thus endow the space
$V$ with the scalar product:
\[[| \by_1, \by_2 |] =  [| v_1, v_2
|]_{V_1} + S[| B_1, B_2 |]_{V_2}.\]
Having defined the spaces $H$ and $V$, we have the dense, continuous and compact embedding:

$V \hookrightarrow H = H^{'}\hookrightarrow V^{'}$.

The next step consists in defining the operators $\bA$ and $\bB$ that appear in (\ref {mhde1}).
$\bA$ is defined through a bilinear coercive form, and $\bB$ by means of a trilinear continuous form.

Define a function $a: V\times V \to {\bf R}$ as follows:
\begin{equation}
a(\by_1, \by_2) = \frac{1}{R_e}[| v_1, v_2|]_{V_1} +
\frac{S}{R_m}[|B_1, B_2|]_{V_2}\label {a}
\end{equation}
The function $a$  defined by (\ref {a}) is continuous and coercive.
 By the Lax-Milgram lemma, there exists an operator $\bA: V \to
V^{'}$ such that $a(\by, \bz) = <\bA \by,\bz>_{V^{'} \t V}$ for all
$\by, \bz \in V$.

\sni
Therefore  $\bA : V \to V^{'}$ can be restricted to a self-adjoint operator
$\bA : {\mathcal D}(\bA) \to H$. 

\sni
We can write
${\mathcal D}(\bA) = {\mathcal D}({\bA}_1)\t {\mathcal D}({\bA}_2)$ where $\bA_1\,\,, \bA_2\,\,,
{\mathcal D}({\bA}_1),\,\, \text{and}\,\,{\mathcal D}({\bA}_2)$
 are obtained as follows:

Consider the ``Stokes" problem in $G$:
\[
-\frac{1}{R_e}\Delta v + \n p = g\quad{\rm  with}\quad \n .v = 0,\,\,\,\, {\rm and}\,\,\,\, v|_{\p G} = 0.
\]
This can be written as $\frac{1}{R_e}[|v, \varphi|]_{V_1} = (g, \varphi)$ for $\varphi \in V_1$.

\noindent
Define $\bA_1 v = \frac{1}{R_e}[| v, \cdot |]_{V_1} = g
\in  V_1^{'}$.
By  the Cattabriga regularity theorem,
$  {\mathcal D}({\bA}_1) = H^2(G) \cap V_1$ whenever $g\in H_1$.

\sni
To define $\bA_2$, consider the elliptic problem in $G$:
\[
\frac{1}{R_m} {\rm curl}\,\,{\rm curl}\, B = g\quad{\rm  with}\quad \n\cdot B = 0;\,\, B\cdot  {\bf n}|_{\p G} = 0;
 \,\,{\rm curl}\, B|_{\p G} = 0.
\]
so that $\bA_2 B =\frac {1}{R_m} {\rm curl}\,\,{\rm curl}\, B = g$.
It is shown in \cite {SerTem} that if $g\in H$, then
\[{\mathcal D}({\bA}_2) = \{ u\in H^2(G): \n \cdot u = 0, u \cdot {\bf n}|_{\p G} = 0 \,\,{\rm and}\,\,
{\rm curl}\, u |_{\p G} = 0\}.\]
The next step consists in defining $\bB$ that figures in (\ref {mhde1}). 

Consider the trilinear
form
$
b(\cdot, \cdot, \cdot)\colon V\t V\t V \to {\bf R}$ by
\begin{equation}
 b(\by_1, \by_2, \by_3) = {\mathaccent 126 b}(v_1, v_2, v_3) - S{\mathaccent 126 b} (B_1, B_2, v_3) + S{\mathaccent 126 b} (v_1, B_2, B_3)
 - S{\mathaccent 126 b} (B_1, v_2, B_3)\label {b}
\end{equation}
 for all $\by_i = (v_i, B_i) \in V$ where ${\mathaccent 126 b} (\cdot, \cdot, \cdot)\colon (H^1(G))^{\otimes 3} \to {\bf R}$ is defined by
\[ {\mathaccent 126 b} (\varphi, \psi, \theta) = \displaystyle{\sum_{i,j = 1}^2 \int_G {\varphi_i\frac {\p \psi_j}{ \p x_i} \theta_j}\,dx}.\]
The function $b$ defined by (\ref {b}) is continuous.
Note that for all $\varphi\in H^1(G)$ and $\psi \in V_1$,
\begin{equation*}
 {\mathaccent 126 b} (\varphi, \psi, \psi) = {\displaystyle{\sum_{i,j} \,\int_G \, \varphi_i\,\frac{\p \psi_j}{ \p x_i}\, \psi_j}\,dx}
=  0
\end{equation*}
by integration by parts.
Also, ${\mathaccent 126 b} (\varphi, \psi, \theta) = - {\mathaccent 126 b} (\varphi, \theta, \psi)$ if $\psi$ or $\theta \in V_1$ and
 $\varphi \in H^1(G)$.
Therefore $b(\by, \bz, \bz) = 0$.

Define $\bB\colon V\t V \to V^{'}$ as the continuous bilinear
operator such that $ b(\by_1, \by_2, \by_3) = < \bB(\by_1, \by_2),\,
\by_3
>$ for all $\by_1, \by_2, \by_3 \in V$. The
 existence of such  an operator is guaranteed by the Riesz representation theorem.
With $\bB(\by)$ denoting $\bB(\by, \by)$, the  stochastic MHD system  appears as follows:
\begin{equation}
 d\by + [\bA\,\by + \bB (\by)]\,dt = f(t)dt +  \int_{\mathbb Z} g(\by(t-), z) \tilde{N}^{1/\e}(dt, dz).\label{sde}
\end{equation}
with $\by(0) \in H$. The hypotheses ${\bf H}$ follow from the above description, and is akin to the two-dimensional Navier-Stokes system with added complexities that arise from the coupling with the Maxwell equation. 

Physically, the two-dimensional MHD system pertains to the case when the domain is a cylinder $G\times \R$ with $G$ as a 
domain in $\R^2$, and  all the quantities are independent of $z$. Thus $v$ and $B$ are parallel to the $xy$-plane.

\medskip
\noindent
\leftline{{\sc 5.  Two-dimensional Tidal Equation}}
Let $G$ be an open domain in $\R^2$ with a smooth boundary $\p G$, and fix any $T > 0$. 
Consider the stochastic system:
\begin{align}
\frac{\p u}{\p t} +\bA u + \bB (u) + \nabla z &= f(t) + \Gamma(t, u) \,\,\text{in}\,\,[0, T] \times G \label{eqn1.1}\\
 \,\,\frac{\p z}{\p t} + \nabla \cdot (hu) &=0  \,\,\text{in}\,\, [0, T]\ 
\times G \label{eqn1.2}
\end{align}
with the initial values $u(0, x) = u_0(x)$ and $z(0, x) = z_0(x)$ are specified. The boundary condition $u(t, x) = 0$ is satisfied for all $ t \in [0, T]$ and $x \in \p G$.

The coefficients $\bA$ and $\bB$ are given by 
\begin{equation}
\bA = \begin{pmatrix} -a\Delta & -\beta\\ \beta & -a \Delta\end{pmatrix}
\end{equation}
and 
\begin{equation}
\bB(u) = \beta |u + w^0(t, x)|(u + w^0)
\end{equation}
where $w^0(t, x)$ is a known random function, and $\beta(x) = \frac{r}{h(x)}$ with $r > 0$. The function $h$ is continuously differentiable such that 
\begin{equation}
\k = \min_G h(x), \mu_1 = \max_G h(x), \mu_2 = \max_G |\nabla h(x)|
\end{equation}
where $\k, \mu_1, \mu_2$ are positive constants. The function $f$ is an additive forcing term. 
The noise term on the right side of \eqref{eqn1.1} is given by 
\[
\Gamma(t, u) = g(t, u, z) \tilde{N}(dt, dz)
\]
where  $g$ satisfies the properties listed for the noise coefficient in hypotheses $\bf H$. $\tilde N$ is a compensated Poisson random measure with compensator  given by $\nu(dz) dt$ for a $\sigma$-finite Borel measure $\l$ defined on $\Xi \in {\mathcal B}(L^2(G))$.  We will call $L^2(G)$ as $H$. 
The $H$-norm is denoted by $| \cdot |$, and the inner product by $(\cdot , \cdot )$.  We will denote the space $H_0^1(G)$  by $V$. The $V$-norm and inner product are denoted by $\norm{\cdot}$ and 
$(( \cdot ))$. Let $V\pr := \mathbb{H}^{-1} (G)$ denote the dual of $V$, so that we have the dense, continuous embedding 
\[
 V \subset_\to H \subset_{\to} V\pr.
\]
The dual pairing between $V\pr$ and $V$ is denoted by $\langle \cdot , \cdot \rangle$.

Consider the nonlinear operator $F$ on $H_0^1(G)$ defined by $F(u) :=  \bA u + \bB (u) - f$. Let $u, v \in V$ and $w:= u- v$.  Then, it is simple to check that for any $\e$ satisfying 
$0 < \e < \frac{\a \k}{2L}$, and $\rho := \displaystyle{\frac{\a (\mu_2^2 + \k^2)}{2\k}}$, we have 
\begin{equation}
\langle F(u) - F(v), hw \rangle - - \rho |w|_{L^2}^2 + \e \int_{Z} |g(t, u, z) - g(t, v, z)|^2 \nu(dz)  \le 0. \notag
\end{equation}
It is easy to verify that the coefficients satisfy hypotheses $\bf H$. 

\section{Weak Convergence Approach to LDP and MDP}

We summarize the weak convergence approach to LDP  as developed by Budhiraja and Dupuis (see \cite{BuDu19}). For a locally compact Polish space $\mathbb{Z}$,  let $\mathcal{M}_{FC}(%
\mathbb{Z})$ be the space of all (non-negative) Radon measures $\nu $ on $(\mathbb{Z},\mathcal{B}(%
\mathbb{Z}))$ (i.e., $\nu (K)<\infty $ for every compact $K\subset 
\mathbb{Z}\times [0,T]$). Endow $\mathcal{M}_{FC}(\mathbb{Z}\times [0,T])$ with the usual vague
topology, which makes $\mathcal{M}_{FC}(\mathbb{Z}\times [0,T])$ a Polish space \cite{BDM09}. 

 For $M>0$, define $S^M := \{\phi: \mathbb{Z} \times [0,T] \to [0, \infty): L_T(\phi) \le M\}$. A function $\phi \in S^M$ can be identified with a measure $\nu_T^{\phi} \in \mathcal{M}_{FC}(\mathbb{Z}\times [0,T])$, defined by
\[
 \nu_T^{\phi}(B) = \int_B \phi(s, x) \nu(dx) ds, \;\;\forall \;\;B \in \mathcal{B}(\mathbb{Z} \times [0,T]). 
\]
Thus the usual topology of weak convergence of measures induces a topology on $S^M$ through this identification  under which $S^{M}$ is a compact space.  Notice that $\mathbb{S} = \cup_{M \ge 1}S^M$ (c.f. \eqref{eq:defS}).

Now let $\mathcal{P}$ be the predictable $\sigma$-field on $ [0,T] \times \Omega $, that is, the $\sigma$-field generated by left-continuous adapted processes. Let $\bar{\mathcal{A}}, \bar{\mathcal{A}}_{+}$ respectively denote the set of maps (viewed as ``controls")  $\vphi: \mathbb{Z}\times  [0,T] \times \Omega  \to \R$ and $\vphi: \mathbb{Z}\times  [0,T] \times \Omega  \to [0,\infty)$ which are $\mathcal{B}(\mathbb Z)\times \mathcal{P} / \mathcal B([0,\infty))$-measurable. Let $\bar{\mathcal{A}}_{b,+} \subset \bar{\mathcal{A}}_{+}$ consists of those $\vphi \in \bar{\mathcal{A}}_{+}$  which are bounded and identically 1 outside a compact set. Specifically, let $\left\{
K_{n}\subset \mathbb{Z},n=1,2,\ldots \right\} $ be an increasing sequence of
compact sets such that $\cup _{n=1}^{\infty }K_{n}=\mathbb{Z}$, and define $\bar{\mathcal{A}}_{b,+} = \cup_{n=1}^\infty \bar{\mathcal{A}}_{b,n,+}$ where 
\begin{align*}
\bar{\mathcal{A}}_{b,n,+}& \doteq \left\{ \varphi \in \bar{\mathcal{A}}_{+}:%
\text{ for all }(z,t,\omega )\in \mathbb{Z}\times  [0,T] \times \Omega\text{, }%
n\geq \varphi (z,t,\omega )\geq 1/n\right. \\
& \hspace{1in}\left. \text{ if }z\in K_{n}\text{ and }\varphi (z,t,\omega )=1%
\text{ if }z\in K_{n}^{c}\right\},
\end{align*}%

For a (random) control $\vphi \in \bar{\mathcal{A}}_{b,+}$ and $\theta>0$, let $N^{\theta \vphi}$ be a ``controlled'' Poisson measure with intensity measure  $\theta \vphi(z,t, \omega) \nu(dz) dt,$ and we assume $(\Omega, \mathcal{F}, \{\mathcal{F}_t\}, \mathbb{P})$ supports all these processes. Mathematically, we can construct any $N^{\theta \vphi}$ on  $(\Omega, \mathcal{F}, \{\mathcal{F}_t\},  \mathbb{P})$ as follows: let $\bar N$ be a PRM on $\mathbb Z \times [0,\infty)\times [0,T]$ with mean measure $\nu\ot \leb\ot \leb$, and define the counting process $N^{\theta \vphi}$ on $\mathbb Z \times [0,T]$ by 
\begin{equation}
N^{\theta \varphi}( U \times (0,t])=\int_{U\times\lbrack0,\infty)\times [
0,t]}1_{[0,\theta\varphi(z,s)]}(r)\bar{N}(dx\,dr\,ds),\quad t\in\lbrack0,T],U\in 
\mathcal{B}(\mathbb{Z}).  \label{Eqn: control}
\end{equation}
Of course, $N^{\ep^{-1}}$ appearing in \eqref{epsiloneqn} corresponds to  \eqref{Eqn: control} with $\theta = \ep^{-1}$ and $\vphi \equiv 1$.



\paragraph{A General Large Deviation Result from \cite{BDM09}:}

For each $\ep>0$, let $\mathcal{G}^{\e}: \mathcal{M}_{FC}(\mathbb{Z}\times [0,T]) \to \mathbb{U}$ be a measurable map. We now state a general theorem for obtaining LDP of $\{Z^{\e} \equiv \mathcal{G}^{\e}(\ep N^{\ep^{-1}})\}$

\begin{condition} \label{LDPcond}
There exists a measurable map $\mathcal{G}^0: \mathbb{M} \to \mathbb{U}$ such that the following hold:

\begin{enumerate}
 \item[(a)] For any fixed $M \in \mathbb{N}$, let $\varphi_n, \varphi \in S^M$ such that $\varphi_n \to \varphi$ as $n \to \infty$. Then 
\[
\mathcal{G}^0(\nu_T^{\varphi_n}) \to \mathcal{G}^0(\nu_T^{\varphi}) \;\,\text{in}\;\,\mathbb{U}.
\]
\item[(b)] For any $M \in \mathbb{N}$, let $\varphi_{\e}, \varphi \in \mathcal{U}^M_+ \doteq \{\varphi \in \bar{\mathcal{A}}_{b,+}: \varphi(\o) \in S^M, \, \mathbb{P}\; a.s.\}$ such that $\varphi_{\e}$ converges in distribution to $\varphi$ as $\e \to 0$. Then 
\[
 \mathcal{G}^{\e}(\e N^{\e^{-1}\varphi_{\e}}) \implies \mathcal{G}^0(\nu_T^{\varphi}).
\]
\end{enumerate}
\end{condition}
For $\bar{\bu} \in \mathbb{U}$, define $\mathbb{S}^{\mathrm{LD}}_{\bar{\bu} } = \{\varphi \in \mathbb{S}: \bar{\bu}  = \mathcal{G}^0(\nu_T^{\varphi})\}$. With $L_T$ as in \eqref{Ltdef}, define $I^{\mathrm{LD}}: \mathbb{U} \to [0, \infty]$  by 
\begin{equation}
 I^{\mathrm{LD}}(\bar{\bu} ) = \inf_{\varphi \in \mathbb{S}^{\mathrm{LD}}_{\bar{\bu} }} L_T(\varphi), \;\;h \in \mathbb{U}.
\label{LDPrate}
\end{equation}
By convention, $I^{\mathrm{LD}}(\bar{\bu} ) = \infty$ if $\mathbb{S}^{\mathrm{LD}}_{\bar{\bu} } = \emptyset$.

\begin{theorem}\cite{BDM09}
Suppose that Condition \ref{LDPcond} holds. Then $I^{\mathrm{LD}}$ defined by 
\eqref{LDPrate} is a rate function on $\mathbb{U}$ and the family $\{Z^{\e} \equiv \mathcal{G}^{\e}(\ep N^{\ep^{-1}}) \}$ satisfies an LDP with rate function $I^{\mathrm{LD}}$.
\end{theorem}
We are interested in proving LDP for a class of SPDEs which in integral form is given by  
\begin{equation}
\bu(t) + \int_0^t [\bA \bu (s) + \bB (\bu(s))] ds = \bu (0) + \int_0^t f(s)ds + \e \int_0^t \int_{\mathbb Z} g(s, \bu(s-), z) \tilde{N}^{\e^{-1}}(ds, dz)
\end{equation}
with $\bA,  \bB$ and $g$ satisfying hypotheses ${\bf H.1} - {\bf H.7}$ and 
with $\bu(0)$  in $H$,  $f \in L^2([0, T]; V\pr)$. 

Next, we state some results that will be useful for establishing LDP for $\bu^\ep$ and refer the reader to \cite{BCD} and \cite{BDM08} for the proofs. 
\begin{lemma}\label{MLlemma}
Suppose that hypotheses ${\bf{H.5}}$ and ${\bf{H.6}}$ hold.
Then we have, 
\begin{align}
(i) \,\,\sup_{\varphi \in S^M} \int_0^T \int_{\mathbb Z} (M_g^2(z) + L_g^2(z))(\varphi(t, z) + 1)\nu(dz)dt & < \infty. \label{mlcond1}\\
(ii) \,\, \sup_{\varphi \in S^M} \int_0^T \int_{\mathbb Z} (M_g(z) + L_g(z))| \varphi(t, z) - 1|\nu(dz)dt & < \infty. \label{mlcond2}
\end{align}
(iii) \,\, For any given $\eta > 0$, there exists a $\d > 0$ such that for all $A \subset [0, T]$ satisfying 
$|A| < \delta$, we have 
\begin{equation}
\sup_{\varphi \in S^M} \int_A \int_{\mathbb Z} (M_g(z) + L_g(z))|\varphi(t, z) - 1| \nu(dz)dt < \eta. \label{mlcond3}
\end{equation}
\end{lemma}
\begin{lemma}\label{hlemma}
Suppose that 
\[
\int_0^T \int_{\mathbb Z} |h(s, z)|^2 \nu(dz)ds < \infty,
\]
 and for all $\d > 0$, and any Borel set $B$ in $[0, T]\times {\mathbb Z}$ with  $\nu_T(B)  < \infty$, 
 \[
 \int_B e^{\d |h(s, z)|} \nu(dz)ds < \infty.
 \]
 Then the following hold:
 \begin{enumerate}
 \item{(i)} If $\varphi_n \to \varphi$ in $S^M$, then 
 \begin{equation}
 \int_0^T \int_{\mathbb Z} h(s, z)(\varphi_n(s, z) - 1)\nu(dz)ds \to \int_0^T \int_{\mathbb Z} h(s, z)(\varphi(s, z) - 1)\nu(dz)ds. \label{h1}
 \end{equation}
 \item{(ii)} Given $\e > 0$, there exists a compact set $K_{\e} \subset \mathbb Z$ such that 
 \begin{equation}
 \sup_{\varphi \in S^M}\int_0^T \int_{K_{\e}^c} |h(s, z)| \,|\varphi(s, z) - 1| \nu(dz)ds < \e. \label{h2}
 \end{equation}
 \item{(iii)} For any compact $K \subset \mathbb Z$, 
 \begin{equation}
  \sup_{\varphi \in S^M}\int_0^T \int_{K} |h(s, z)|1_{\{ |h(s, z)| \ge R\}}\, \varphi(s, z)  \nu(dz)ds \ \xrightarrow{R \to \infty}\ 0. \label{h3}
 \end{equation}

 \end{enumerate}
\end{lemma}

\paragraph{ A General Moderate Deviation Result from \cite{BDG}}

\label{Sec:thm} For any $\epsilon>0$, let $\mathscr{G}^{\epsilon}$ be a
measurable map from $\mathbb{M}$ to $\mathbb{U}$, where $\mathbb{U}$ is a
Polish space. Let $a:\mathbb{R}_{+}\rightarrow (0,1)$ be such that as $%
\ep\rightarrow0$ 
\begin{equation}
a(\ep)\rightarrow0\mbox{ and }b(\ep)=\frac{\ep}{%
a^{2}(\ep)}\rightarrow0.  \label{eq:eq1138}
\end{equation}
In this section we will formulate a general sufficient condition for the
collection $\mathscr{G}^{\ep}(\ep N^{\ep^{-1}})$ to
satisfy a large deviation principle with speed $b(\ep)$ and a rate
function that is given through a suitable quadratic form.

For $\ep>0$ and $M<\infty$, consider the spaces 
\begin{align}
\mathcal{S}_{+,\ep}^{M} & \doteq\{\varphi:\mathbb{Z}\times
\lbrack0,T]\rightarrow\mathbb{R}_{+}: \ L_{T}(\varphi)\leq
Ma^{2}(\ep)\}  \label{ctrlclass}, \\
\mathcal{S}_{\ep}^{M} & \doteq\{\psi:\mathbb{Z}\times\lbrack
0,T]\rightarrow\mathbb{R}:\ \psi=(\varphi-1)/a(\ep),\varphi \in%
\mathcal{S}_{+,\ep}^{M}\}.  \notag
\end{align}
We also let 
\begin{align}
{\mathscr{U}}_{+,\ep}^{M} & \doteq\left\{ \varphi\in\bar{\mathcal{A}}%
_{b}:\varphi(\cdot,\cdot,\omega)\in\mathcal{S}_{+,\ep}^{M},\text{ }%
\mathbb{\bar{P}}\text{-a.s.}\right\}  \label{eq:def_of_Us}, \quad
{\mathscr{U}}_{\ep}^{M}  \doteq\left\{ \psi\in\bar{\mathcal{A}}%
:\psi(\cdot,\cdot,\omega)\in\mathcal{S}_{\ep}^{M},\text{ }\mathbb{%
\bar{P}}\text{-a.s.}\right\} .  \notag
\end{align}

The norm in the Hilbert space $L^{2}(\nu_{T})$ will be denoted by $\norm{
\cdot }_{2}$ and $B_{2}(r)$ denotes the closed ball of radius $r$ in $L^{2}(\nu_{T}) $. 

 A set $\{\psi^{\ep}\}\subset%
\bar {\mathcal{A}}$ with the property that $\sup_{\ep>0}\|\psi
^{\ep}\|_{2}\leq M$ a.s. for some $M<\infty$ will be regarded as a
collection of $B_{2}(M)$-valued random variables, where $B_{2}(M)$ is
equipped with the weak topology on the Hilbert space $L^{2}(\nu_{T})$. Since 
$B_{2}(M)$ is weakly compact, such a collection of random variables is
automatically tight. Throughout this paper $B_{2}(M)$ will be regarded as
the compact metric space obtained by equipping it with the weak topology on $%
L^{2}(\nu_{T})$.

Suppose $\varphi \in \mathcal{S}_{+,\ep }^{M}$.  Then as shown in  Lemma  
3.2  in \cite{BDG}, there exists $\kappa _{2}(1)\in (0,\infty )$ that is
independent of $\ep $ and such that $\psi 1_{\{|\psi
|<1/a(\ep )\}}\in B_2((M\kappa_2(1))^{1/2})$ where 
$\psi = (\varphi - 1)/a(\e)$. 

We now state an LDP result that is crucial  for establishing  MDP for a family of random elements (recall that MDP entails large deviation analysis in a different scaling regime.) 

For each $\ep>0$, let $\mathscr{G}^{\e}: \mathcal{M}_{FC}(\mathbb{Z}\times [0,T]) \to \mathbb{U}$ be a measurable map.
\begin{condition}
\label{gencond} There exists a measurable map $\mathscr{G}_{0}:L^{2}(\nu
_{T})\rightarrow \mathbb{U}$  such that the following  conditions hold.\label{cond1}
		
\begin{enumerate}
\item[(a)] Given $M\in (0,\infty )$, suppose that $h^{\ep },h\in
B_{2}(M)$ and $h^{\ep }\rightarrow h$. Then 
\begin{equation*}
\mathscr{G}_{0}(h^{\ep })\rightarrow \mathscr{G}_{0}(h).
\end{equation*}
			
\item[(b)] Given $M\in (0,\infty )$, let $\{\varphi ^{\ep
}\}_{\ep >0}$ be such that for every $\ep >0$, $\varphi^{\ep }\in \mathscr{U}_{+,\ep }^{M}$ and for some $\beta \in (0,1]$, $\psi ^{\ep }1_{\{|\psi ^{\ep }|\leq \beta
/a(\ep )\}}\Rightarrow \psi $ in $B_{2}(M')$ for some constant $M'>0$ and 
where $\psi ^{\ep }=(\varphi ^{\ep }-1)/a(\ep )$.
Then 
\begin{equation*}
\mathscr{G}^{\ep }(\ep N^{\ep ^{-1}\varphi^{\ep }})\Rightarrow \mathscr{G}_{0}(\psi ).
\end{equation*}%
\label{cond2}
\end{enumerate}
\end{condition}

For $\bar{\by} \in \mathbb{U}$, define $\mathbb{S}^{\mathrm{MD}}_{\bar{\by} } = \{\psi\in L^{2}(\nu_{T}): \bar{\by}=\mathscr{G}_{0}(\psi)\}$, and define $I^{\text{MD}}: \mathbb{U} \to [0, \infty]$  by 
\begin{equation}
I^{\mathrm{MD}}(\bar{\by})=\inf_{\psi\in\mathbb{S}^{\mathrm{MD}}_{\bar{\by}}}\left[ \frac{1}{2}\|\psi\|_{2}^{2}%
\right].  \label{eq:first_rate}
\end{equation}

\begin{theorem}\label{LDPpoi} \cite{BDG} Suppose that the functions $\mathscr{G}^{\ep }$ and $%
\mathscr{G}_{0}$ satisfy Condition \ref{gencond}. Then $I^{\mathrm{MD}}$ defined by (\ref{eq:first_rate}) is a rate function, and $\{Y^{\ep }\equiv \mathscr{G}^{\ep }(\ep N^{\ep ^{-1}})\}$ satisfies an LDP with speed $b(\ep )$ and the rate function $I^{\mathrm{MD}}$.
\end{theorem}

Finally, we state some results that will be needed for proving MDP of $\bu^\ep$. 

For any fixed $\delta > 0$, let
\[
\mathcal{H}^{\d} = \{h: \mathbb{Z} \to \mathbb{R}: \int_A \exp\{\d h^2(z)\} \nu(dz) < \infty \,\,\text{for all}\,\, A \in \mathcal{B}(\mathbb{Z}) \,\,\text{with}\,\, \nu(A) < \infty\}.
\]
The next three Lemmas appear as Lemma 3.2, Lemma 4.3 and Lemma 4.4 in 
\cite {BDG}.
\begin{lemma}\label{lemma1}
Let $\varphi \in \mathcal{S}_{+, \e}^M$ for some finite $M$, and $\psi = (\varphi - 1)/a(\e)$. Then, there exist 
positive-valued functions $\kappa_1, \kappa{\pr}_1$ and $\kappa_2$ defined on $(0, \infty)$ such that 
\begin{enumerate}
\item[1.] $\int_0^T \int_{\mathbb Z} |\psi |1_{\{|\psi | \ge \beta/ a(\e)\}} \nu(dz)ds \le Ma(\e)\kappa_1(\beta),$ 
\item[2.] $\int_0^T \int_{\mathbb Z} \varphi 1_{\{\varphi \ge \beta\}} \nu(dz)ds \le Ma^2(\e)\kappa{\pr}_1(\beta)$ if $\beta > 1$, and
\item[3.] $\int_0^T \int_{\mathbb Z} |\psi |^21_{\{|\psi | \le \beta/ a(\e)\}} \nu(dz)ds \le M \kappa_2(\beta).$ 
\end{enumerate}
The functions $\kappa_1(\beta), \kappa{\pr}_1(\beta)$ tend to $0$ as $\beta \to \infty$, and $\kappa_2$ is non-decreasing.
\end{lemma}

\begin{lemma}\label{lemma2}  
Let $h \in L^2(\l) \cap \mathcal{H}^{\d}$ and $I$ be a measurable subset of $[0, T]$. Fix $\zeta, M > 0$. 
Then there exist  functions $\Gamma_h,\, \rho_h, \,\theta: (0, \infty) \to (0, \infty)$ such that $\Gamma_h(u) \to 0$ as 
$u \to \infty$,  and $\theta(u) \to 0$ as 
$u \to 0$,  such that for all positive $\e$ and $\beta$, 
\begin{equation}
\sup_{\psi \in S_{\e}^M} \int_{\mathbb Z \times I} | h(z) \psi(z, s)|1_{\{|\psi | \ge \beta/a(\e)\}} \nu(dz)ds \le \Gamma_h( \beta) a^{1 - \zeta}(\e) + (1 + | I |) \theta(\e)\label{l2est1}
\end{equation}
and
\begin{equation}
\sup_{\psi \in S_{\e}^M} \int_{\mathbb Z\times I} |h(z) \psi(z, s)| \nu(dz)ds \le \rho_h(\beta) \sqrt{| I |}  + \Gamma_h( \beta) a^{1 - \zeta}(\e) + (1 + | I |) \theta(\e). \label{l2est2}
\end{equation}
\end{lemma}
\begin{lemma}\label{lemma3} 
Let $h \in L^2(\l) \cap \mathcal{H}^{\d}$ be positive valued. Then, for any $\beta > 0$, 
\[
\lim_{\e \to 0} \sup_{\{\psi \in S_{\e}^M\}} \int_{\mathbb Z \times [0, T]} h(z) |\psi(z, s)|1_{\{|\psi | \ge \beta/a(\e)\}}
\nu(dz)ds = 0.
\]
\end{lemma}
The next Lemma is a particular case of Lemma 4.5 in \cite{BDG}.
\begin{lemma}\label{lemma4}
Let $h \in L^2(\l) \cap \mathcal{H}^{\d}$. Fix any $M > 0$. Then, there exists a constant $k_h > 0$ such that 
for any measurable set $I$ subset of $[0, T]$ and for all $\e > 0$, 
\[
\sup_{\varphi \in S_{+, \e}^M} \int_{\mathbb Z \times I}  h^2(z)\varphi(z, s)\nu(dz)ds \le k_h(\delta, M, |I|)a^2(\e) + M\delta^{-1}| I |.
\]
\end{lemma}


\end{document}